\documentclass[12pt]{paper}
\usepackage{graphicx}
\usepackage{subfig}
\usepackage{amsthm}
\usepackage{amsfonts}
\usepackage{amsmath}
\usepackage[makeroom]{cancel}

\newtheorem{proposition}{Proposition}
\newtheorem{theorem}{Theorem}

\newcommand{\pder}[2]{\frac{\partial #1}{\partial #2}}
\newcommand{\R}{\mathbb{R}}
\title{Global Bifurcation Diagram for the Kerner-Konh\"auser Traffic Flow Model}

\author{Joaqu\'\i n Delgado \and Patricia Saavedra}

\begin{document}
\maketitle


\begin{abstract}
We study  traveling wave solutions of the Kerner--Konh\"auser PDE for traffic flow. By a standard change of variables, the problem is reduced to a dynamical system in the plane with three parameters.
In a previous paper \cite{Ca1} it was shown that under general hypotheses on the fundamental diagram, the dynamical system has a surface of critical points showing  either a fold or cusp catastrophe when projected under a two dimensional plane of parameters named $q_g$--$v_g$.  In any case a one parameter family of Bogdanov--Takens (BT) bifurcation takes place, and therefore local families of Hopf and homoclinic bifurcation arising from each BT point exist. Here we prove the existence of a degenerate Bogdanov--Takens bifurcation (DBT) which in turn implies the existence of Generalized Hopf or Bautin bifurcations (GH). We describe numerically the  global lines of bifurcations continued from the local ones, inside a cuspidal region of the parameter space. In particular, we compute the first Lyapunov exponent, and compare with the GH bifurcation curve. We present some families of stable limit cycles which are taken as initial conditions in the PDE leading to stable traveling waves.

\end{abstract}

\section{Introduction}
Macroscopic traffic models are posed in analogy to continuous one dimensional, compressible  flow. Second-order  models consist of a system  of two coupled equations involving  the density $\rho(x,t)$ and the average velocity $V(x,t)$. In the  Kerner--Konh\"auser model these variables are related through   the continuity  and momentum equation
\begin{eqnarray}
&&\frac{\partial \rho}{\partial t}+\frac{\partial \rho V}{\partial x}=0, \label{continuity}\\
&&      \rho\left(\pder{V}{t}+V\pder{V}{x}\right)= -\pder{P}{x}+\frac{\rho (V_e(\rho)-V)}{\tau}. \label{balance}
\end{eqnarray}
 Here in analogy with compressible fluids,  the rate of change in momentum in (\ref{balance}) is
 due to a decreasing  gradient in ``pressure" $P$. The bulk forces are modeled as a tendency to acquire
 a safe velocity $V_e(\rho)$. The constant $\tau$ is a relaxation time. The model can be closed
 by a constitutive equation of the form
 $$
 P=\rho\Theta-\eta\pder{V}{x},
 $$
 where $\Theta(x,t)$ is the traffic ``variance" and $\eta$  is the analogous of the viscosity. Here and in what follows
 we will take $\Theta(x,t)=\Theta_0,$ and $\eta=\eta_0$ as  positive constants.   See \cite{KK0} for details.

 The fundamental diagram is the  relationship between  the average velocity    and traffic density $V=V_e(\rho)$.  Although empirical data shows that even the mere existence of such a functional relationship may be criticized~\cite{KSS}, we depart from the point of view that it yields a first approximation by assuming
 homogeneous solutions where the density and the average velocity remain constant but  are related through the fundamental diagram.

Next in complexity are traveling wave solutions. Under the
 change of variables $\xi=x+V_g t$ system  (\ref{continuity})--(\ref{balance}) is transformed into a system of ordinary differential equations.
 In the process of integration of the continuity equation (\ref{continuity}), there appears    the  constant $Q_g$  having the dimension of flux. In this paper  $\Theta_0$, $Q_g$   and $V_g$  are considered as the main parameters of the present study.
 The first one has a dynamical character being the proportional factor among density and pressure, $-V_g$ describes the velocity of the traveling wave and $Q_g$  is the net flux as measured by an observer moving with the same velocity as the wave \cite{Saa-Ve}.

 The main motivation for doing this research is to analyze if bounded solutions of the dynamical system can give us valuable information  of the system of PDEs (\ref{continuity})--(\ref{balance}) for different boundary conditions:   periodic for a finite domain, or bounded  for an infinite domain.  Others authors such as Lee, Lee and Kim \cite{Lee} have work with the dynamical system relating, in a qualitative form, its  solutions  to solutions of the PDE.  As far as we know,  this is the first time in this context that the dynamical machinery is applied in order to make a rigorously analysis of the global bifurcation diagram, and establishing a relation between what is observed in the dynamical system, and the solutions of the PDE.

We have shown in a previous work \cite{Ca1}, that under general properties of the fundamental diagram,
a one parameter curve of   Takens-Bogdanov (BT)  bifurcations exists, associated to a folding projection of the surface of critical points into  the two--dimensional space of parameters $Q_g$--$V_g$.  The family of BT points can be parametrized by the value of $\Theta_0$. For a fixed value of $\Theta_0$ the versal unfolding of the BT point contains codimension--one local curves of Hopf and homoclinic bifurcations in the $Q_g$--$V_g$ plane.

In this article we consider  the  dynamical system for a particular fundamental diagram due to Kerner and Konh\"auser:
\begin{equation}\label{FD}
V_e(\rho)=V_{max}\left(\frac{1}{1+\exp{[(\frac{\rho}{\rho_{max}}-0.25)/0.06]} }-3.72 \times 10^{-6}\right).
\end{equation}
We  compute explicitly  the bifurcation set and show that there exists a cuspidal curve in the parameter space $Q_g$--$V_g$ corresponding to BT bifurcations for a proper choice of $\Theta_0$.

The main result refers to the cuspidal point  of the bifurcation curve. We show that this  is in fact a degenerate Takens--Bogdanov point  (DBT), whose bifurcation diagram corresponds to the saddle case, according to Dumortier et al in \cite{Du}. We also prove that
a local curve of GH bifurcations originates from DBT and that a bifurcation of two limit cycles can occur in our model (one stable and the other unstable) \textit{for the same values of the parameters}.
 We also compute the first Lyapunov exponent $\ell_1$
and describe the set of GH points as the zero set  $\ell_1=0.$ This defines a curve that divides limit cycles bifurcating from Hopf curves into stable and unstable.
We use systematically  Kusnetzov and Govaert's \textit{Matcont } in order to  perform the global numerical continuation of Hopf bifurcation and limit cycles curves that gives the global picture of bifurcations. We take as initial conditions for system (\ref{balance}) two limit cycles, generated by Matcom, one in the stable region other in the unstable, and we show that they give place to two traveling waves that can be stable or unstable.

The rest of the paper is organized as follows: in Section ~2 we introduce the dynamical system,  and the surface of critical points where we give conditions for non hyperbolic points to be Hopf or Takens Bogdanov.
 In Section~3 we present all the theoretical results, including the calculation of the first Lyapunov exponent in order to analytically determine the curve of Bautin points,  or Generalized Hopf points which let us determine the stability region of  limit cycles,  associated to Hopf points. We also show the existence of a degenerate Takens Bogdanov bifurcation.

 In Section~4  we present  the dynamical consequences of the global bifurcation diagram obtained in the previous sections. This includes families of homoclinic an heteroclinic solutions.
 In Section~5 we study in detail families of limit cycles which represent periodic traveling waves of the PDE in a bounded domain.
 Finally, conclusions are given in Section~6. At the end of the article we include the proof of some of the theoretical results.

\section{The dynamical system and the surface of critical points}\label{TRsection}
We look for traveling wave solutions of (\ref{continuity},\ref{balance}). In order to obtain it we apply to these equations the following change of variables  $\xi=x+V_g t.$ The first equation is transformed into a quadrature which can be immediately solved:
\begin{equation}\label{flux}
\rho (V+V_g)=Q_g.
\end{equation}

Following \cite{Saa-Ve} we introduce dimensionless variables
\begin{equation}\label{adim}
    z=\rho_{max}\xi,\quad v=\frac{V}{V_{max}},\quad v_g=\frac{V_g}{V_{max}},\quad q_g=\frac{Q_g}{\rho_{max}V_{max}},\quad r= \frac{\rho}{\rho_{max}}.
\end{equation}
 Then (\ref{flux}) becomes
\begin{equation}\label{flux2}
   r=\frac{q_g}{v+v_g},
\end{equation}
and observe that in the fundamental diagram (\ref{FD}), $V_e$ depends only on the ratio $r$. By abuse of notation we also write $V_e(\rho)$ as $V_e(r)$.
Also let
\begin{equation}\label{adim2}
    \tilde{v_e}(r)=\frac{V_e(r)}{V_{max}},\quad \theta_0=\frac{\Theta_0}{V_{max}^2},\quad \lambda=\frac{V_{max}}{\eta_0},\quad
    \mu=\frac{1}{\rho_{max}\eta_0\tau}.
\end{equation}
In what follows we will denote by $v_e(v)$ the composition of $\tilde{v_e}$  with $r$ given by (\ref{flux2}),
and whenever we want to make explicit the dependence on the parameters
\begin{equation}\label{ve}
v_e(q_g,v_g,v)=\tilde{v_e}\left(\frac{q_g}{v+v_g}\right).
\end{equation}
Also for simplicity in the notation  we will use the shorthand
$$v_e'(v)=\pder{v_e(q_g,v_g,v)}{v}.
$$
Observe that $v_g$ and $v$ appear symmetrically in (\ref{ve}), therefore
$$
\pder{v_e(q_g,v_g,v)}{v_g}=\pder{v_e(q_g,v_g,v)}{v}=v_e'(v).
$$

Substitution of (\ref{flux}) into the second equation of (2) yields the following
dynamical system
\begin{eqnarray}\label{KKode}
  \frac{dv}{dz} &=& y, \nonumber\\
  \frac{dy}{dz} &=& \lambda q_g \left[1-\frac{\theta_0}{(v+v_g)^2}\right] y-\mu q_g
  \left(\frac{v_e(v)-v}{v+v_g}\right).
\end{eqnarray}
Here and in what follows, we will take the parameter values $\lambda$, $\mu$ as given by the model, and we will
analyze the dynamical behavior with respect to the parameters $\theta_0$, $v_g$, $q_g.$

\begin{proposition}
 Let $V_e(\rho)$ be given by (\ref{FD}) then there exist parameter values for $q_g$ and $v_g$ such that the dynamical system has up to 3 critical points.
\end{proposition}
This proposition was proved in \cite{Ca1}.    The Figure~\ref{cusp} shows the corresponding graph for the Kerner--Konh\"auser fundamental diagram in the case of three critical points.
\begin{figure}
\centering
\includegraphics[height=2in]{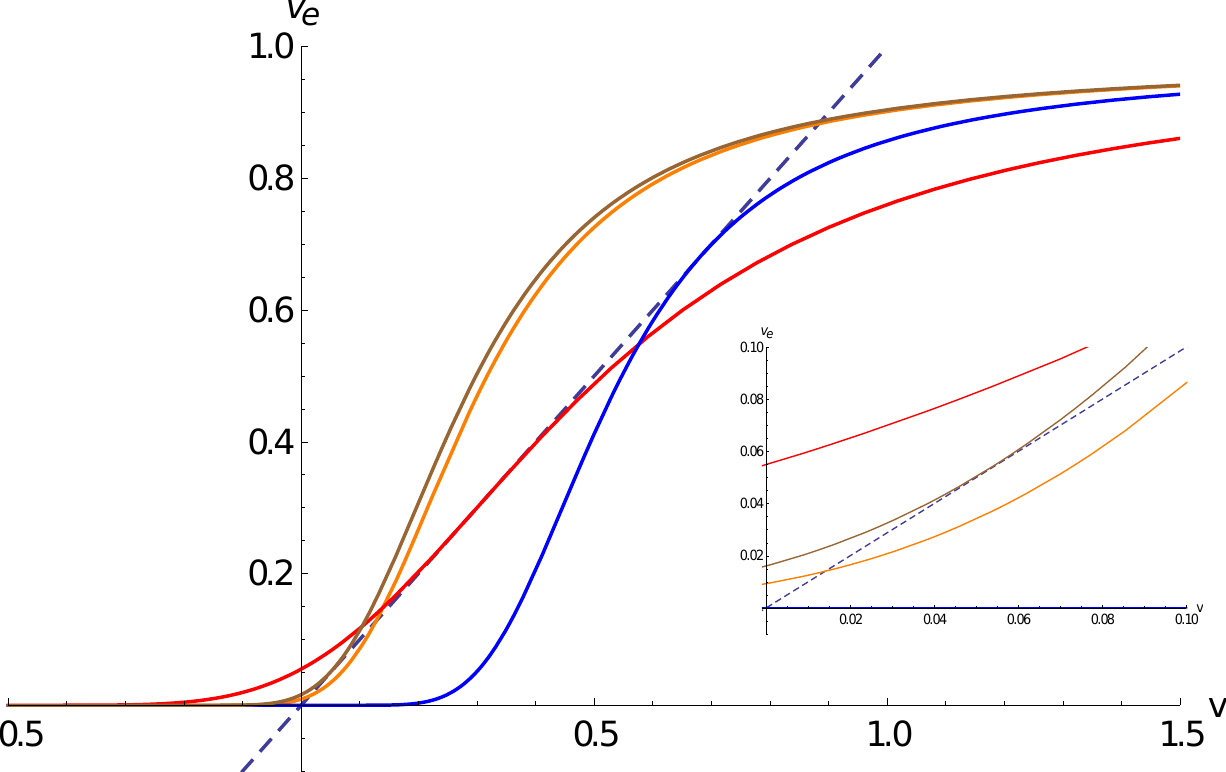}
\caption{Kerner-Konh\"auser fundamental diagram $v_e(v)$ showing  up to three  intersections with the identity (dashed line): $v_e(v_c)=v_c$. Distinct situations are illustrated by graphs in different colors. Red: $v_e'(v_c)=v_e''(v_c)=0$.  Brown and blue:  $v_e'(v_c)=0$. Orange: three intersections, the middle one with $v_e'(v_c)>1$, the others satisfy $v_e'(c_c)<1.$}
\label{sigmoide}
\end{figure}
The linear part of (\ref{KKode}) at $v_c$ is
$$
A_0=\left(
\begin{array}{ll}
 0 & 1 \\
 -\frac{\mu q_g    \left(v_e'(v_c)-1\right)}{v+v_g} &
 \lambda q_g  \left(1-\frac{\theta_0}{(v_c+v_g)^2}\right)
\end{array}
\right)\equiv  \left(
\begin{array}{ll}  0 & 1 \\ c & b
\end{array}
\right).
$$
The characteristic polynomial $\lambda^2-b\lambda -c=0$ yields
the eigenvalues
\begin{equation}\label{eigenvalues}
l_{1,2}=\frac{b\pm \sqrt{b^2+4c}}{2}.
\end{equation}
The stability of the critical points is given in the following proposition  \cite{Ca1} .
\begin{proposition}\label{stability}

 Let $(v_c,0)$ be a critical point of system (\ref{KKode}), then
\begin{itemize}
  \item If $v_e'(v_c)<1$ then $c>0$ and the roots $l_{1,2}$ are real and with opposite signs. Thus the critical point is a saddle.
  \item If $v_e'(v_c)>1$ then $c<0$ and  either the roots $l_{1,2}$ are real of the same sign as $b$ and the critical point is a node, or $l_{1,2}$ are complex conjugate with real part $b$ and the critical point is a focus. Thus the sign of  $b$ determines the stability
  of the critical point: if $b<0$ it is stable, if $b>0$ it is unstable.
  \item If $v_e'(v_c)=1$ then $c=0$ and one eigenvalue becomes zero. If in addition, $b=0$ then  zero is
  an eigenvalue of multiplicity two.
\end{itemize}
\end{proposition}

 Whenever there are  three critical points, two of them  $v_c^1<v_c^2$  are saddles,  and one  is  a stable/unstable focus or node $v_c$ depending on the parameter values $(q_g,v_g)$, and $v_1^2<v_c<v_c^2$. In this case the condition $v_e'(v_c)>1$ must be satisfied.

For the Kerner--Konh\"auser fundamental diagram  (\ref{FD}) the set of critical points is given by the surface
\begin{equation}\label{fold_surface}
\{(q_g,v_g,v_c)\mid v_e(v_c)-v_c=0\},
\end{equation}
which is depicted in Figure~\ref{cusp}.
\begin{figure}[ht]
\centering
  \includegraphics[width=2in]{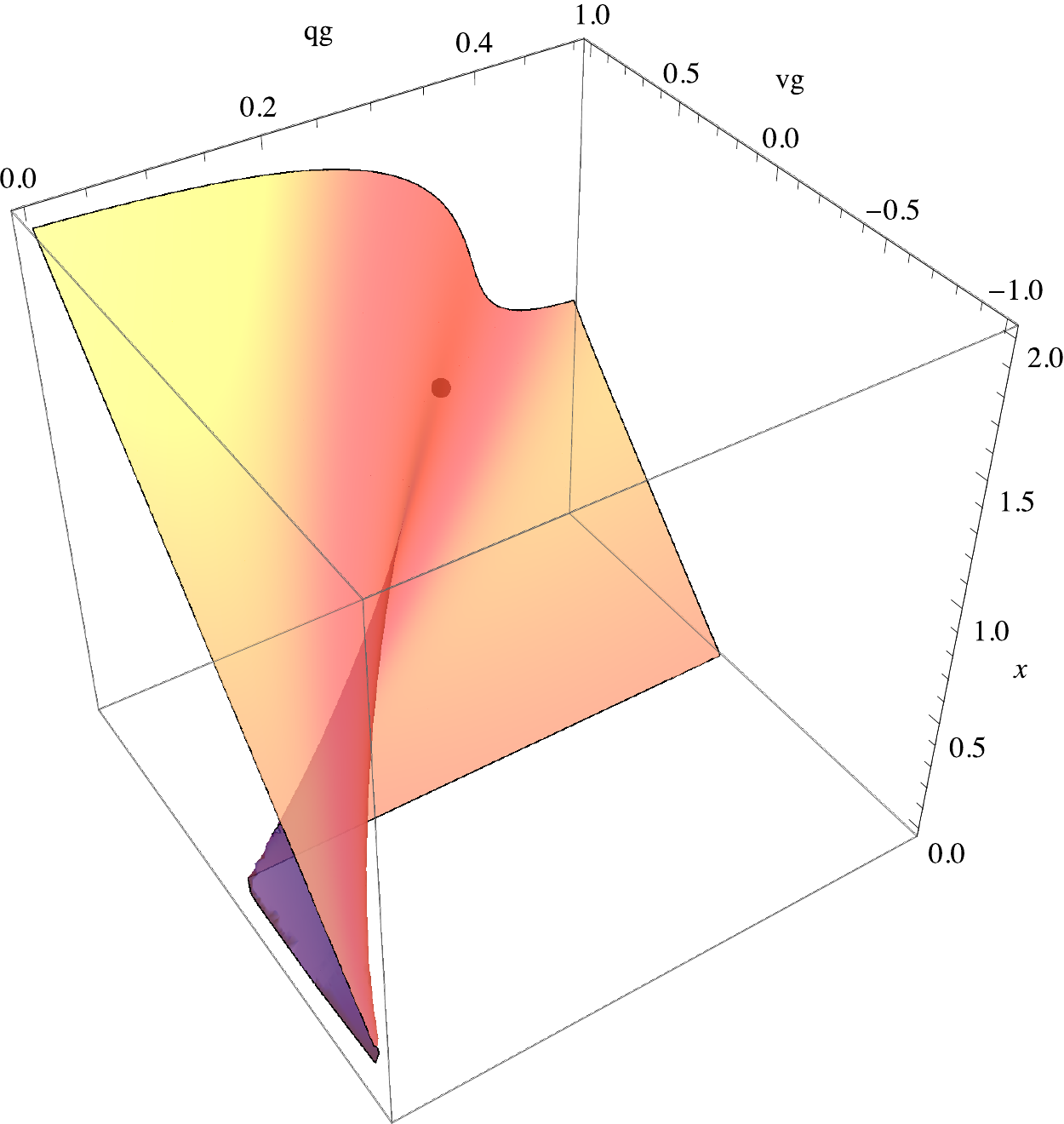} \quad
   \includegraphics[width=2in]{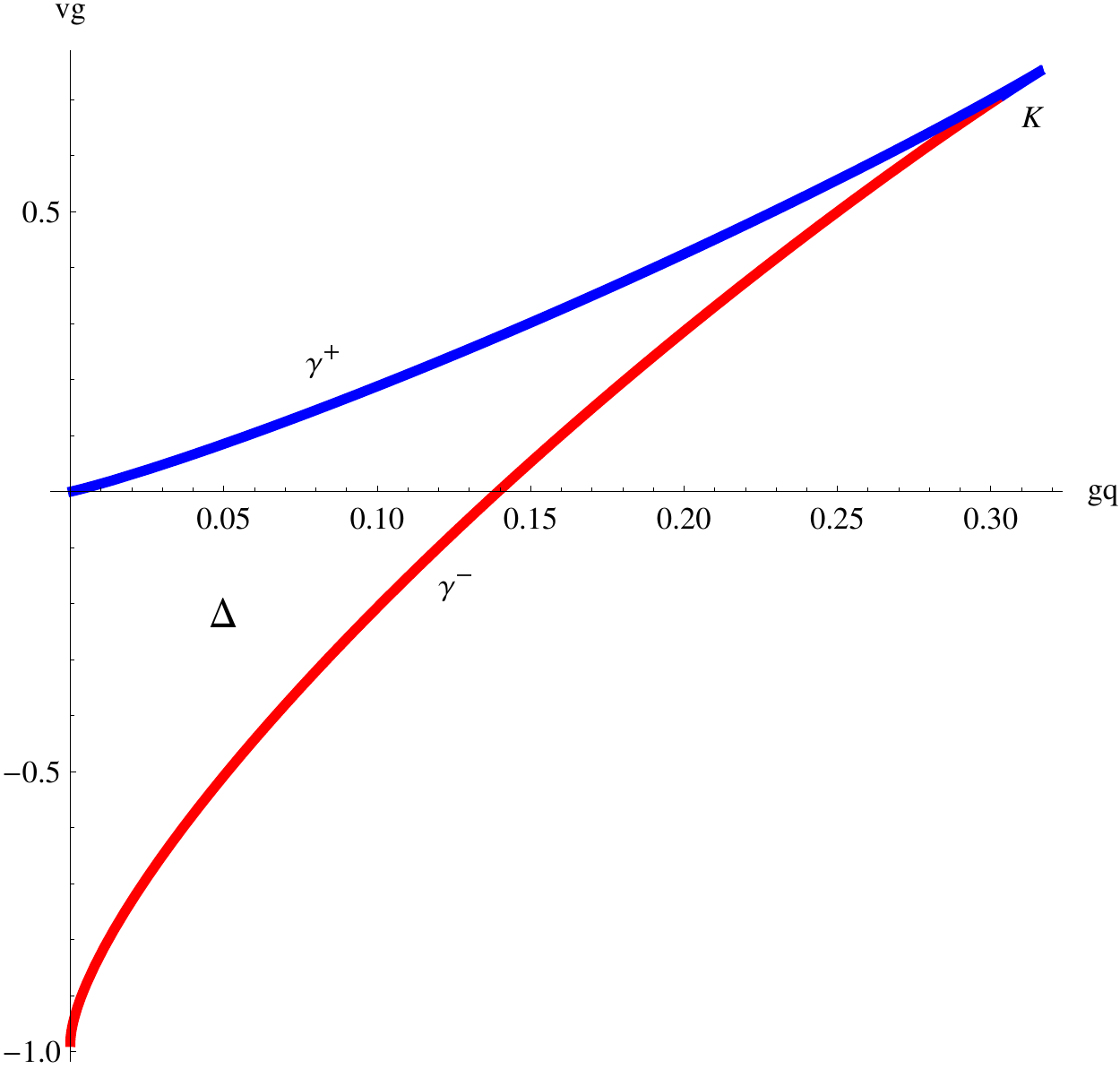}
  \caption{Left: Surface of critical points. Right:  The singular locus of the projection $\gamma$. The upper part  $\gamma^+$ is shown in blue, the lower part $\gamma^-$ in red.}
  \label{cusp} \end{figure}
  For simplicity,
the surface of critical points is represented in $(q_g,v_g,x)$ coordinates
where $x=v_g+v_c$ and we restrict to $x>0$.
Geometrically for given points in the parameter plane  $(q_g,v_g)$,  the coordinates of the critical points
$v_g+v_c$  are obtained as intersections of the line parallel to the $x$--axis passing through the point.

There is a curve in three dimensional space $(v_g,q_g,v_c)$ where the surface of critical points folds back. It is the set of points where the projection $(v_g,q_g,v_c)\stackrel{\pi}{\to} (v_g,q_g)$ restricted to the surface fails to be surjective.
Analytically, this set is a curve  given  by two  equations
\begin{eqnarray*}
\tilde{\gamma}=\{ (q_g,v_g,v_c)\mid v_e(v_c)-v_c= 0,\quad v_e'(v_c)-1=0\}.
\end{eqnarray*}
This curve and its projection $\gamma=\pi(\tilde{\gamma})$  in parameter space $q_g$--$v_g$
are shown in Figure ~\ref{cusp}.  For
  $(q_g,v_g)\in\gamma $ the graph of  $v_e(v)$ is tangent
to the  identity at $v_c$ which is then a saddle--node. If in addition,
 $\theta =\sqrt{v_g+v_c}$  then the critical point is a Takens--Bogdanov bifurcation point  whenever the non--degeneracy conditions
 \begin{equation}\label{ND}
 v_e''(v_c)\neq 0,\qquad\mbox{and}\qquad
\frac{\partial^2 v_e(v_c)}{\partial q_g\partial v}\neq 0
\end{equation}
are satisfied.

The complement of $\tilde{\gamma}$ has two components, the folded part corresponds to critical points such that  $v_e(v_c)=v_c$ and  $v_e'(v_c)>1$. This follows from the sigmoidal shape of the curve $v_e(v)$, shown in Figure~\ref{sigmoide}, see \cite{Ca}. The second component contains the saddle points associated to the same value of the parameters $(q_g,v_g)$  where $v_e'(v_c)<1.$

The cusp point  $K$ of the curve $\gamma$ is defined by the three conditions
\begin{equation}\label{degenerate_BT}
v_e(v_c)-v_c=0,\qquad v_e'(v_c)=1,\qquad v_e''(v_c)=0,\quad v_e'''(v_c)\neq0
\end{equation}
and divides $\gamma$ in two components. We will call $\gamma^+$  the upper, and $\gamma^-$ the  lower part of $\gamma$ according to Figure~\ref{cusp}.
It will be analyzed   in detail in Section 3.2, that this cusp point gives rise to a degenerate Takens--Bogdanov (DTB) bifurcation.
Here we just mention that for the Kerner--Konh\"auser fundamental diagram (\ref{FD}) there exists a unique point $(q_v^*,v_g^*,v_c^*)$ satisfying (\ref{degenerate_BT}) with $v_e'''(v_c^*)<0$, therefore $K=(q_g^*,v_g^*)$. Numerical values are given in Section 4.1.

\section{Global bifurcations inside the cusp}
In this paper we will be interested in the cuspidal region $\Delta$ with boundary  $\gamma=\partial\Delta$, which is the projection of the patch of the surface that folds back:
\begin{equation}\label{fold_surface}
\mathcal{F}=\{(q_g,v_g,v_c)\mid v_e(v_c)-v_c=0,\quad v_e'(v_c)>1\}.
\end{equation}

\begin{proposition}
Let  $\pi_{\mathcal{F}}$ be the restriction of the projection $(q_g,v_g,v_c)\mapsto (q_g,v_g)$ to $\mathcal{F}$. Then $\pi_{\mathcal{F}}\colon\mathcal{F}\to\Delta $ is a diffeomorphism .
\end{proposition}

\begin{proof} Let $p^{(0)}=(q^{(0)}_g,v_g^{(0)})\in\Delta$.
By the implicit function theorem applied to $v_e(q_g,v_g,v_c)-v_c=0$,  if  $v_e'(v_c)>1$ there exists  a smooth function $\kappa_{p_0}$, defined in a neighborhood $\mathcal{N}_{p_0}$ of $p^{(0)}$, such that
$v_e(q_g,v_g,\kappa_{p_0}(q_g,v_g))-\kappa_{p_0}(q_g,v_g)=0$, for $(q_g,v_g)\in \mathcal{N}_{p_0}$.  Obviously $\Delta=\bigcup_{p\in\Delta}\mathcal{N}_{p}$. Let the map $k\colon\Delta\to\mathcal{F}$
be defined by $k(q_g,v_g)=\kappa_{p_0}(q_g,v_g)$ if $(q_g,v_g)\in\mathcal{N}_{p_0}$. We will see that $k$ is well defined. For this, suppose $(q_g,v_g)\in \mathcal{N}_{p_1}\cap \mathcal{N}_{p_2}$. By contradiction,  suppose
$\kappa_{p_1}(q_g,v_g)\neq \kappa_{p_2}(q_g,v_g)$.
 Then
$v_e(q_g,v_g,\kappa_{p_i}(q_g,v_g))=\kappa_{p_i}(q_g,v_g),\quad i=1,2$, and by the mean value theorem
\begin{eqnarray*}
\kappa_{p_1}(q_g,v_g)-\kappa_{p_2}(q_g,v_g)&=&
v_e(q_g,v_g,\kappa_{p_1}(q_g,v_g))-v_e(q_g,v_g,\kappa_{p_2}(q_g,v_g))\\
 &=&v_e'(q_v,v_g,v_c)\left(\kappa_{p_1}(q_g,v_g)-\kappa_{p_2}(q_g,v_g)\right)
\end{eqnarray*}
Thus
\begin{eqnarray*}
|\kappa_{p_1}(q_g,v_g)-\kappa_{p_2}(q_g,v_g)|&=&
     |v_e'(q_v,v_g,v_c)| |\kappa_{p_1}(q_g,v_g)-\kappa_{p_2}(q_g,v_g)|\\
&>&|\kappa_{p_1}(q_g,v_g)-\kappa_{p_2}(q_g,v_g)|.
\end{eqnarray*}
this completes the proof.
\end{proof}

 For future reference we compute by implicit differentiation
\begin{equation}\label{partial_vg}
\pder{v_c}{v_g}=-\frac{v_e'(v_c)}{v_e'(v_c)-1}.
\end{equation}

In the following section we present the global picture of bifurcations  appearing in system (\ref{KKode}) inside the cuspidal region $\Delta$.  We will describe the global Hopf curves emerging from Takens-Bogdanov, and the families of limit cycles which  originated in Hopf points. We also compute the first Lyapunov coefficient which determines their stability (see Proposition~\ref{FirstLyapunov}).  When the first Lyapunov coefficient is zero we get a curve of Bautin bifurcations (see Section 3.1) . We also show that  the cuspidal point is a degenerate Takens--Bogdanov point whose bifurcation diagram corresponds to the saddle case  studied by Dumortier et al~\cite{Du}. This prove rigorously the existence of Bautin bifurcations.

\subsection{Bautin bifurcation}
Generalized Hopf or  Bautin bifurcation has  codimension two. Its normal form is given in  \cite[p. 311]{Kuz} and  its bifurcation diagram is  shown in Figure~\ref{Bautin}. For our purposes it will be enough to  recall that necessary conditions can be stated in terms of the eigenvalues $l_{1,2}=\mu(\alpha)\pm i\omega(\alpha)$ depending on the vector of parameters
$\alpha\in\R^2$, namely
\begin{equation}\label{Bautin-cond}
\mu(0)=0,\quad \ell_1(\alpha)=0.
\end{equation}
Additional non--degeneracy conditions involving the second Lyapunov coefficient $\ell_2(0)$,  and the regularity of the map $\alpha\mapsto (\mu(\alpha),\ell_1(\alpha))$ are shown to be sufficient.

We will prove the  existence of this kind of bifurcation, indirectly, by proving that in fact a codimension three bifurcation, a degenerate Takens--Bogdanov, occurs associated to a cusp point of the surface of bifurcation. See Theorem~\ref{DTB},  and  in the Appendix~B its proof. Thus  the existence of Bautin bifurcations will follow from the normal form  already mentioned~\cite{Du}. In this way we will not need to verify explicitly the non--degeneracy conditions.

 The bifurcation diagram  for a Bautin bifurcation is shown in Figure~\ref{Bautin}. It†contains two branches of subcritical ($H_{+}$) and supercritical  ($H_{-}$) Hopf bifurcations and  a single branch of saddle--node bifurcation of cycles LPC (standing for limit point of cycles) where two hyperbolic stable and unstable cycles, coalesce in a single saddle--node  cycle.

\begin{figure}[ht]
\centering
  \includegraphics[width=2in]{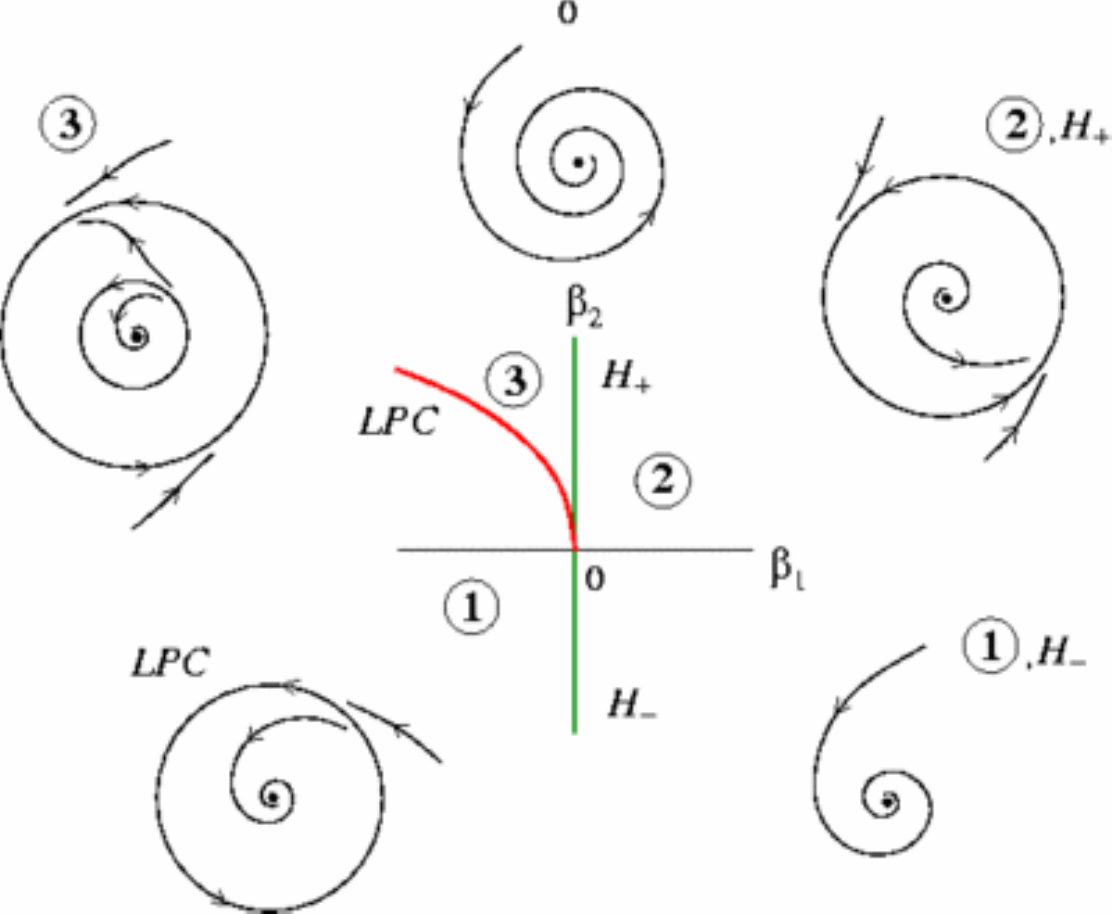}
  \caption{Bifurcation diagram for Bautin bifurcation}
  \label{Bautin}
  \end{figure}

Besides the vanishing of the First Lyapunov coefficient $\ell_1$ determines  a Bautin deformation, its sign also  determines the stability of a limit cycle emerging from a Hopf bifurcation.  The explicit form of $\ell_1$ is stated in Proposition~\ref{FirstLyapunov},  and it will be of great importance in the numerical study of limit cycles presented in Section \ref{numerical}.

Given  $q_g,$ $v_g,$ denote by $l_{1,2}(q_g,v_g)=\mu(q_g,v_g)\pm\omega(q_g,v_g) i$ the
eigenvalues (\ref{eigenvalues}) of the  linearization.

 Let $(v_c,0) $ be a critical point of (\ref{KKode}) such that $v_e(v_c)=v_c$, $v_e'(v_c)>1$ and choose $\theta_0=(v_c+v_g)^2$,  then $b=0$  and the eigenvalues are purely imaginary
 $$
 l_{1,2}=\pm i\omega_0
 $$
 with
  $$
 \omega_0^2=\frac{\mu q_g(v_e'(v_c)-1)}{(v_c+v_g)}.
 $$

\begin{proposition}\label{FirstLyapunov} Let $(v_c,0)$ be a critical point such that $v_e'(v_c)>1$,  and $\theta=\sqrt{v_v+v_g}$, then
the first Lyapunov coefficient is given by the expression
\begin{equation}\label{ell1}
\ell_1(q_g,v_g)=-\frac{\lambda\mu q_g^2}{2\omega_0^3(v_c+v_g)^2}\left( \frac{v_e'(v_c)-1}{v_c+v_g}+v_e''(v_c)\right).
\end{equation}
\end{proposition}
The proof is a straightforward computation and is presented in the Appendix~A.

\begin{proposition}

There exists a smooth  function $v_g=h(q_g)$ defined for $0<q_g<q_g^*$ such that $\ell_1(q_g,h(q_g))=0$ and\, $\lim_{q_g\to q_q^*}h(q_g)=v_q^*$. In other words, $\ell_1(q_g,v_g)=0$ is the graph of a function that divides $\Delta$ and has limit point at
 $K=(q_g^*,v_g^*)$, the cusp point of the curve $\gamma$.
\end{proposition}
\begin{proof}
Observe that from definition (\ref{ve}) it follows that
\begin{equation}\label{derivs}
\pder{v_e(v_c)}{v_g}=\left(1+\pder{v_c}{v_g}\right)v_e'(v_c),\quad \pder{v_e'(v_c)}{v_g}=
\left(1+\pder{v_c}{v_g}\right)v_e''(v_c),
\end{equation}
and so forth. From the expression for $\ell_1$ in (\ref{ell1}) we compute
\begin{eqnarray*}
\lefteqn{\left.\pder{\ell_1(q_g,v_g)}{v_g}\right|_{\ell_1=0}}\\
&=&
-A\left(
\frac{\pder{v_e'(v_c)}{v_g}}{v_c+v_g}+(v_e'(v_c)-1)
\left(-\frac{1}{(v_c+v_g)^2}-\frac{1}{(v_c+v_g)^2}\pder{v_c}{v_g}\right)
+\pder{v_e''(v_c)}{v_g}
\right)\\
&=&
-A\left(
\frac{\pder{v_e'(v_c)}{v_g}}{v_c+v_g}-
\frac{(v_e'(v_c)-1)}{(v_c+v_g)^2} \left( 1+\pder{v_c}{v_g}\right)
+\pder{v_e''(v_c)}{v_g}
\right)
\end{eqnarray*}
where
$$
A= \frac{\lambda\mu q_g^2}{2\omega_0^3(v_c+v_g)^2}
$$
is a positive quantity.
Using (\ref{derivs}) we get
\begin{eqnarray}
\lefteqn{\left.\pder{\ell_1(q_g,v_g)}{v_g}\right|_{\ell_1=0}}\nonumber \\
&=&-A\left(
\frac{v_e''(v_c)}{v_c+v_g}-\frac{v_e'(v_c)-1}{(v_c+v_g)^2}+v_e'''(v_c)
\right)\left( 1+\pder{v_c}{v_g}\right)\nonumber\\
&=&
A\left(
\frac{v_e''(v_c)}{v_c+v_g}-\frac{v_e'(v_c)-1}{(v_c+v_g)^2}+v_e'''(v_c)
\right)\left( \frac{1}{v_e'(v_c)-1}\right),\label{three-dogs}
\end{eqnarray}
where we have used (\ref{partial_vg}).
We now analyze the sign of each term in the second factor:
for the first term, observe that
along $\ell_1=0$,
$$
v_e''(v_c)=-\frac{v_e'(v_c)-1}{v_c+v_g}<0.
$$
The second  term is negative since $v_e'(v_c)-1>0$.
For the third term, recall  that for  fixed values of $q_g$,  $v_g$, $v_e(v)$ is sigmoidal  \cite{Ca}; therefore,  its graph is monotone increasing and the concavity changes from convex to concave passing trough a unique point of inflection. Then the second derivative passes from $v_e''>0$ to $v_e''<0$. Thus $ v_e''(v)$ is decreasing. In particular, $v_e'''(v_c)<0$.
Therefore, the second factor  in (\ref{three-dogs}) is negative.  Since the first and second factors
are negative we conclude that
\begin{equation}\label{dl1}
\pder{\ell_1(q_g,v_g)}{v_g}<0
\end{equation}
whenever $\ell_1(q_g,v_g)=0$.

Define the  Lagrangian
$$
L(q_g,v_g)=-\int_{v_g^0}^{v_g} \ell_1(q_g,s)\,ds
$$
and the associated Legendre transform
$$
\mathcal{L}(q_g,v_g)=(q_g,p),\quad\mbox{where}\quad p=\pder{L}{v_g}(q_g,v_g)
$$
then its is immediate that $\mathcal{L}$ is inyective: If $\mathcal{L}(q_g,v_g)=(q_g',v_g')$ then $q_g=q_g'$ and
$$
\pder{L}{v_g}(q_g,v_g)=\pder{L}{v_g}(q_g,v_g')
$$
that is $\ell_1(q_g,v_g)=\ell_1(q_g,v_g')$; by monotonicity this implies $v_g=v_g'$. The Jacobian determinant of $\mathcal{L}$ is given by
$$
\left|
\begin{array}{cc}
1 & 0 \\
\frac{\partial^2 L}{\partial q_g\partial v_g} & \frac{\partial^2 L}{\partial v_g^2}
\end{array}\right|
=
\frac{\partial^2 L}{\partial v_g^2}=-\ell_1(q_g,v_g)
 >0,
$$
from (\ref{dl1}). Thus $\mathcal{L}$ is a global diffeomorphism onto its image. Let the inverse mapping be denoted as
$$
(q_g,v_g)= (q_g,\mathcal{H}(q_g,p))
$$
then, by definition
$$
p=\ell_1(q_g,\mathcal{H}(q_g,p)),
$$
setting $p=0$ we get
$$
0=\ell_1(q_g,\mathcal{H}(q_g,0)).
$$
This completes the proof by setting $v_g=h(q_g)=\mathcal{H}(q_g,0).$
\end{proof}

\noindent We call $L_1$ the curve defined by $\ell_1(q_g,v_g)=0$.

From the last proposition it follows that  the cuspidal region $\Delta$ is  divided in two components by the graph of  $L_1$. We are now able to determine the regions where $\ell_1>0$ and $\ell_1<0$. 

\begin{proposition}\label{ell1_positive}
The first Lyapunov coefficient $\ell_1(q_g,v_g)$ is positive  for the lower cuspidal region $\Delta^{-}$,
and it is negative  for the upper cuspidal region $\Delta^{+}$.
\end{proposition}
\begin{proof}
From the previous Proposition it follows that
$$
\pder{\ell_1(q_g,v_g)}{v_g}<0.
$$
Take a point $(q_g,v_g)\in L_1$, therefore $\ell_1(q_g,v_g)=0$. Since $ \ell_1(q_g,v_g)$ is decreasing with respect to $v_g$, it follows that $\ell_1(q_g,v_g+\delta)<0$  for small $\delta>0$, but $(q_g,v_g+\delta)\in\Delta^+$ which is connected;  therefore, $\ell_1(q_g,v_g)<0$ for all $(q_g,v_g)\in\Delta^+$.
By continuity of $\ell_1$ in $\Delta$,
$\ell_1$ is positive in $\Delta^{-}$.
\end{proof}

In Figure~\ref{diagram1} the regions $\Delta^{\pm}$  are delimited by the  corresponding curves  
$\gamma^{\pm}$ and $L_1$.
In Figure~\ref{diagram1-left}, the dashed curve interpolates  a number of  points computed numerically with Matcont,  where $\ell_1=0$ (see Section~\ref{numerical}). In Figure~\ref{diagram1-right}, the same set of points and the curve  
$L_1$,  as given by expression (\ref{ell1}), are plotted showing a remarkable fitting.

\subsection{Degenerate Takens-Bogdanov bifurcation}
Among codimension three bifurcation that have been study, degenerate Takens--Bogdanov bifurcation is relevant to this paper.   The monograph of Dumortier, Roussarie,  Sotomayor \& \.Zol\c{a}dek \cite{Du} is the main reference to our work.

 Our presentation follows closely  \cite{Kuz-I}.
  Whenever a system of the form $x'=f(x,\alpha)$, $x,\alpha\in\R^2$, with $f(0,0)=0$, $A=f_x(0,0)$ has a double zero eigenvalue with non--semisimple Jordan form, then the ODE is formally smooth equivalent to
\begin{eqnarray}\label{smooth-equivalent}
\dot{w}_0 &=& w_1,\\
\dot{w}_1 &=&  \sum_{k\geq2}\left(a_k w_0^k + b_k w_0^{k-1}w_1\right).
\end{eqnarray}

In the non--degenerate case $a_1b_2\neq0$, the universal unfolding is  the well known Takens-Bogdanov system. When $a_2=0$ but $a_3b_2\neq0$,   the system
is smoothly orbitally equivalent  to
\begin{eqnarray}\label{orbitally-equivalent}
\dot{w}_0 &=& w_1,\\
\dot{w}_1 &=& a_3 w_0^3+b_2w_0w_1+ b_3' w_0^2 w_1 + O(||(w_0,w_1)||^5).
\end{eqnarray}
There appear three inequivalent cases
\begin{itemize}
\item When $a_3>0$, it is called  the saddle case.
\item When $a_3<0$, $b_2^2+8a_3<0$ and $b_3'\neq0$, it is called focus case.
\item When $a_3<0$ and $b_2^2+8a_3>0$, it is called the elliptic case.
\end{itemize}
According to \cite{Kuz-I} in all cases, a universal unfolding is given by
\begin{eqnarray}
\dot{\xi}_0 &=& \xi_1, \nonumber\\
\dot{\xi}_1 &=& \beta_1+\beta_2\xi_0+\beta_3\xi_1 + a_3\xi_0^3+b_2\xi_0 \xi_1 +b_3' \xi_0^2\xi_1.
\label{UniversalUnfoldingDBT}
\end{eqnarray}
An equivalent bifurcation diagram, after a re-scaling, is presented in \cite{Du}.

\begin{theorem}\label{DTB}
Let $v_c$ a critical point of (\ref{KKode}) that satisfies $v_e(v_c)=v_c,$ $v_e'(v_c)=1,$ $v_e''(v_c)=0$ but $v_e'''(v_c)<0.$ If $\theta_0=(v_c +v_g)^2$ is chosen then this point corresponds to a degenerate Takens-Bogdanov point whose bifurcation diagram is the saddle case.
\end{theorem}

The proof is given in the Appendix~B.

\noindent The bifurcation diagram of the universal unfolding (\ref{UniversalUnfoldingDBT}) is given in Dumorter et al., see~\cite{Du}. A  sketch is shown in Figure~\ref{lips} keeping their notation.
\begin{figure}[ht]
\centering
  \includegraphics[width=2in]{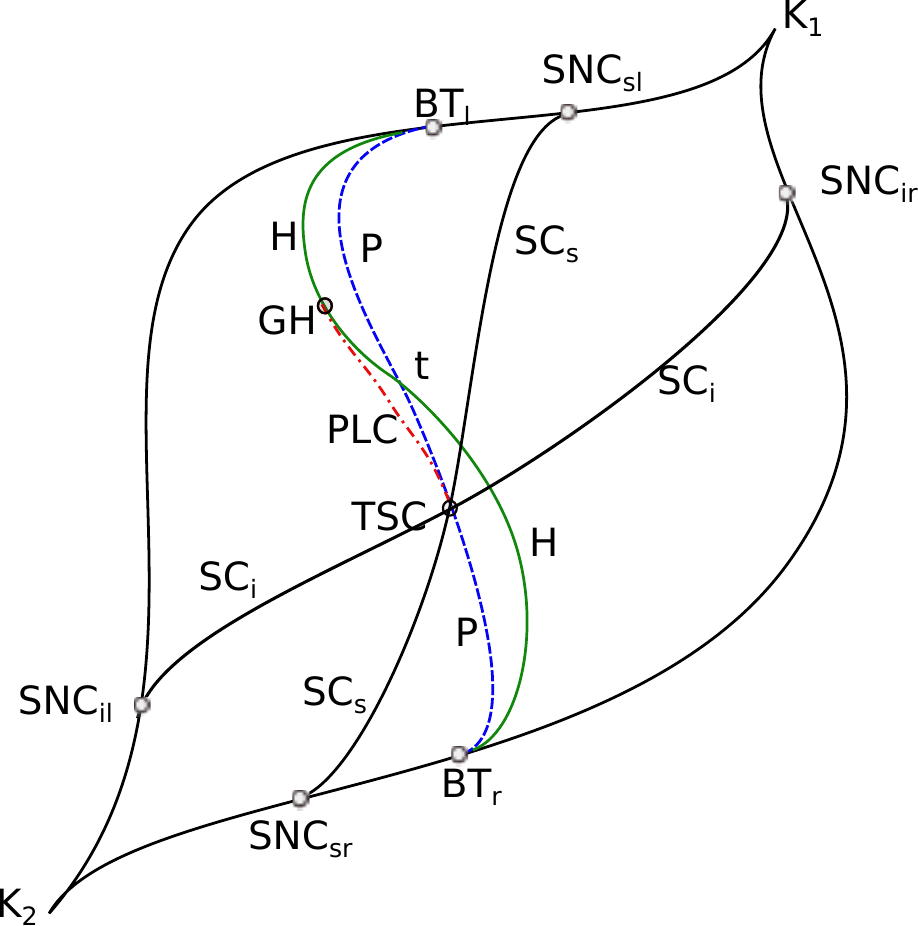}
  \caption{Sketch of the degenerate Takens--Bogdanov bifurcation diagram. Notation is: BT: Takens-Bogdanov; H: Hopf; GH: Bautin; P: homoclinic: PLC: saddle-node  cycle; TSC: two saddle connections; SC: saddle connection; SNC: saddle--node connection. The subindices mean
  s: superior, i: inferior, l: left, r: right, and describe the position of the bifurcation in the phase plane.}
  \label{lips}
  \end{figure}

The description is as follows:  within the  lips--shaped region there exists three critical points, two saddles and an interior focus or node. In the outer part of the lips,  there exists exactly one saddle. The two regions are separated by a closed curve formed either by BT points (if we choose the value of the parameter $\theta_0=\sqrt{v_c+v_g}$) otherwise by saddle--nodes.

If we start with the  (left) Takens--Bogdanov point BTl in the left part of the curve,  there are two branches of homoclinic and Hopf bifurcating from it, according to Takens-Bogdanov theorem \cite{Kuz}. The homoclinic curve of bifurcation  P,  the dotted line in blue,  continues up to a point TSC,   and terminates in a second BTr point, in the right part of the curve.   The BTl point arises when the left saddle in phase space coalesce with the focus/node.  At the BTr point the right saddle coalesce with the focus/node.
 TSC  is also a point of intersection of  two curves bifurcating from  saddle--node connection points, named (superior left) SNCsl  and (inferior right) SNCir, which intersect precisely at TSC. They continue separately ending up at two different saddle--node connection points named (superior right) SNCsr,  and (inferior left) SNCil, respectively. The curve joining the points  SNCsl  and SNCsr is named SCs; the curve  connecting the points  SNCir,  and SNCil is named SCi.    SCs and SCi are curves of saddle---saddle connections, connecting two saddles in  phase space  by a regular curve connecting a saddle point and a saddle-node.

 The Hopf curve of bifurcating from the BT point continues up to a Bautin point named GH (generalized Hopf), and continues as a Hopf curve that ends in the same BT point as the previous described homoclinic curves of bifurcation.

 There is a segment line connecting the TSC point, and the GH point, marked as a dot--line red curve, denoted by LPC. This curve consists of saddle-node cycle. This curve is the same as the local curve LPC in the local diagram of the Bautin bifurcation in Figure~\ref{Bautin}. When we cross PLC from the exterior of the triangular region GH-t-TSC,  an hyperbolic saddle becomes a saddle-node cycle, and bifurcates into two limit cycles --one unstable and the other stable-- just as the local diagram of the Bautin bifurcation in Figure~\ref{Bautin}.

 \section{Dynamical consequences in the PDE}
In order to obtain a particular solution of  system (\ref{continuity}), (\ref{balance}) initial and boundary conditions must be given. Let  $f(x)$, $g(x)$ be smooth functions such that  $V(x,0)=f(x),$ and  $\rho(0,t)=g(x).$ We discuss two types of boundary conditions: (a) periodic in a finite road and, (b) bounded in an unbounded road. More precisely for type (a),    for $0<x<L$ we consider
 the boundary conditions
 \begin{equation}\label{periodic}
 V(0,t)=V(L,t),\quad \rho(0,t)=\rho(L,t)\quad\mbox{for all $t>0$.}
 \end{equation}
 For type (b),  we consider
 the boundary conditions
 \begin{equation}\label{unbounded}
 V(x,t)\quad\mbox{and}\quad  \rho(x,t) \quad\mbox{remain bounded as $x\to\pm\infty$ for all $t>0.$}
 \end{equation}
Of course type (b) boundary conditions can only be approximated numerically by a sufficient long finite road, but they are interesting to discuss for theoretical purposes.

The solutions of interest, arising from the dynamical system (\ref{KKode}),  can be classified according to  the Poincar\'e--Bendixon theorem as:
 \begin{enumerate}
\item Critical points.
\item Limit cycles.
\item Cycles of critical points and homoclinic orbits.
  \begin{enumerate}
  \item Homoclinic connections.
  \item Heteroclinic connections.
  \end{enumerate}
\end{enumerate}

 \subsection{Critical points}
 Critical points, $v_e'(v_c)=v_c$, give rise to homogeneous solutions for both types of boundary conditions (a) and (b).  Under the change of variables (\ref{adim}), critical points  are given by a pair of values $(\rho_0, V_0)$ in the graph of the fundamental diagram: $V_0=V_e(\rho_0)$.  The linear stability is given according to Proposition~\ref{stability} .
 In \cite{Saa-Ve},  type (a) boundary conditions were considered. It was shown, numerically, that if the homogeneous solution is linearly  unstable in the PDE then it evolves, under a small perturbation, into a traveling wave solution. This observed behavior can be partially explained by the dynamical system (\ref{KKode}) as follows: consider an unstable critical point  of the focus type with parameters $(q_g,v_g)$ within  the cuspidal region $\Delta$, surrounded by  a stable limit cycle. This scenario takes place whenever a Hopf bifurcation with negative Lyapunov coefficient takes place. Then by a small perturbation of the initial condition near the critical point, solutions evolve along the unstable spiral towards the stable cycle.

 \subsection{Homoclinic and heteroclinic connections}
 Homoclinic solutions are associated to saddle points, located to the left or right in the $v$ direction of phase space $v$--$y$ of system (\ref{KKode}). This kind of solutions correspond to  one--bump traveling wave solutions with the same horizontal asymptotes as $\xi\to\pm\infty$ (see Figure~\ref{clinics}  left). The family of homoclinic orbits described in Section~\ref{homoclinics} are accumulation points of limit cycles. If it is accumulated by unstable cycles, then the homoclinic presents a ``two sided" stability behavior: it is stable for initial conditions within the annular region defined by  the unstable limit cycle and the homoclinic, but it is unstable for initial conditions outside the limit cycle.  This poses the possibility that an unstable traveling wave would evolve towards a one--bump traveling wave in the PDE by a proper small perturbation.

 If the homoclinic is accumulated by stable limit cycles, then it is always unstable.  Heteroclinic orbits are interpreted similarly, and give rise to traveling fronts as shown in Figure~\ref{clinics}.

 \subsection{Double saddle connection}
This is a  codimension three phenomenon. As explained in the bifurcation diagram of Figure~\ref{lips},  for a fixed value of $\theta_0$,  a double saddle connection  is determined  as the intersection of  two lines of saddle-node (homoclinic) connections. In the PDE there coexist, for the same value of the parameters, two front traveling waves as shown in Figure~\ref{clinics}.
\begin{figure}[htbp]
\begin{center}
\includegraphics[width=2.5in]{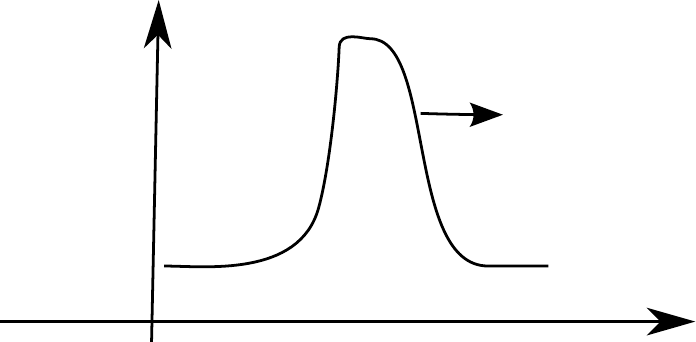}\quad
\includegraphics[width=2.5in]{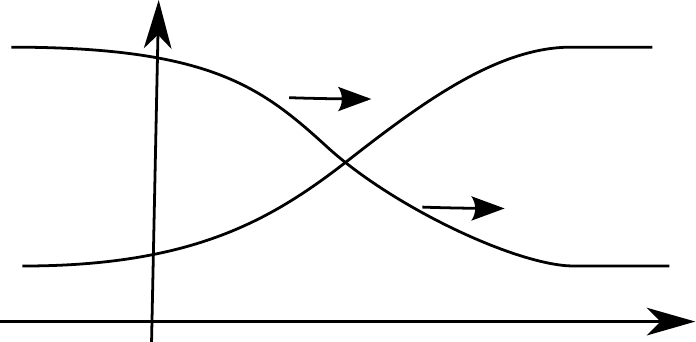}
\caption{One-bump and  coexisting front traveling wave solutions, corresponding to a homoclinic (left) and a double saddle connection (right).\label{clinics}}
\end{center}
\end{figure}

 \subsection{Heteroclinic connection between two limit cycles}
 This type of solutions arise within the triangular region of parameters GH-t-TSC shown in Figure~\ref{lips}, where two limit cycles, one unstable the other stable,   and  the annular region in between contains a double asymptotic spiral. An orbit of this type correspond to an increasing in amplitude oscillating traveling solution as shown in Figure~\ref{twocycles}.
 \begin{figure}[htbp]
\begin{center}
\includegraphics[width=4in]{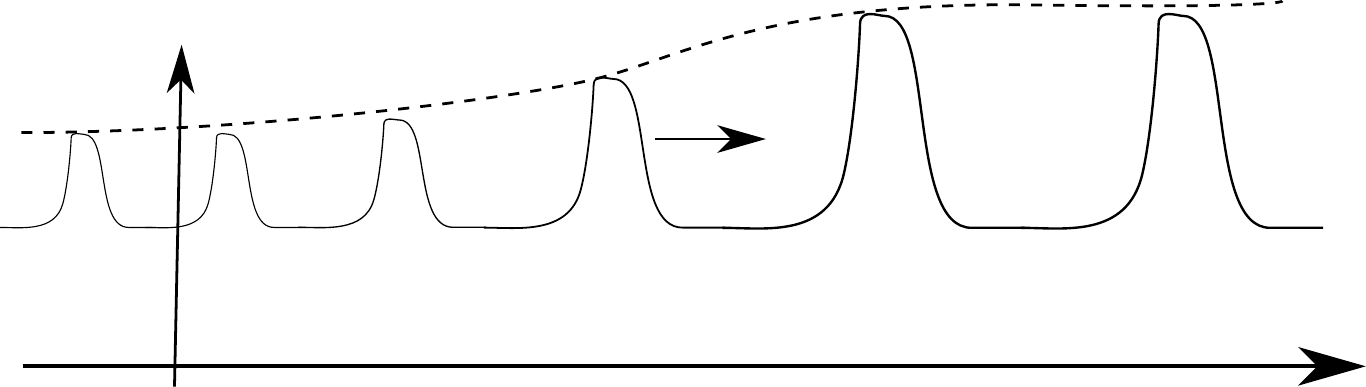}
\caption{Heteroclinic connecting two limit cycles give rise to increasing in amplitude traveling solution.\label{twocycles}}
\end{center}
\end{figure}

\section{Periodic boundary conditions}\label{numerical}
For periodic boundary conditions in a bounded road of length $L$, only periodic solutions of (\ref{KKode})  that satisfy the condition
\begin{equation}\label{resonance}
L\rho_{max}= mT,
\end{equation}
 for some positive integer $m,$  where $T$ is the period of the limit cycle, give rise to  traveling wave solutions,  see \cite{Ca1}.   If $T$ is the minimal period, then we call $L_0=T/\rho_{max}$ the minimal road length. Then
by considering  a limit cycle of minimal period $T$ as a limit cycle of  period $mT$ yields a traveling wave solution in a road of length $mL$. In this way one can obtain multiple bump- traveling waves in the PDE.

The following result characterizes the shape of traveling wave solutions of minimal period in a road of minimal length.
\begin{proposition}
Let $T$ be the minimal period of a limit cycle and consider a road of minimal length $L_0$. Then the corresponding traveling wave solution has exactly one minimum and one maximum.
\end{proposition}
\begin{proof} According to (\ref{KKode})  a limit cycle crosses transversally the $v$--axis exactly twice. These are the minimum and maximum of $v(z)$.
\end{proof}
This result says that traveling wave solutions of minimal period  in a road of minimal length are one--bump traveling waves.

\bigskip
In the rest of the  section we  compute the global bifurcation diagram inside the cuspidal region in the parameter space $q_g$--$v_g$. We use Matcont to extend numerically, the local curves of bifurcations given by the  Takens--Bogdanov theorem, namely Hopf and homoclinic curves.  We present in detail the continuation of limit cycles from Hopf points which give rise to periodic orbits of fixed period,  that correspond to traveling wave solutions in the PDE.  For each BT point we found a GH bifurcation when continuing Hopf curves, that constitute a complete family of Bautin bifurcations, which
 are given by the condition $\ell_1=0$ that is numerically verified.

The presence of Bautin bifurcations found in this study are consistent with the global bifurcation diagram presented in \cite{Du}, and in fact are justified by Theorem \ref{DTB}.

For the Kerner-Kornh\"auser fundamental diagram we use
the following parameter values:
$$
\rho_{max}= 140\, veh/km,\quad V_{max}= 120\, km/h,\quad \tau=30\ seg,\quad \eta_0=600\ km/h,
$$
$$
\lambda=\frac{1}{5}=0.2,\qquad \mu=\frac{1}{700}=0.00142857.
$$

\subsection{Cusp point}
With these values one can show that there exist a unique critical point that satisfies the hypotheses $v_e(v_c)=v_c$, $v_e'(v_c)=1$, $v_e''(v_c)=0$  of Theorem \ref{DTB},  given by
$$
q_g=0.316762381,\quad
v_g= 0.752937578,\quad
v_c = 0.300464598,\quad
\theta=1.109656146,
$$
and $v_e'''(v_c)=-11.317691591012832<0.$

\subsection{The Hopf curves}\label{Hopf_curves}
In Figure \ref{diagram1}, we show the continuation of Hopf curves from several BT points taken on the lower branch of the cuspidal curve. Amid each  continuation, a GH point is found, and we show with a dotted line  the  interpolated curve passing through these points. When $\ell_1(q_g,v_g)=0$ is plotted, a remarkable fit is shown. According to Proposition~\ref{ell1_positive}, the  Hopf points that are located below the GH--curve ($\Delta^{-}$)  have positive Lyapunov coefficient, while those located above ($\Delta^{+}$) have a negative exponent, therefore limit cycles which bifurcate from Hopf points in this region are stable.

\begin{figure}
\centering
 \subfloat[]{\label{diagram1-left} \includegraphics[width=2.5in,height=2.3in]{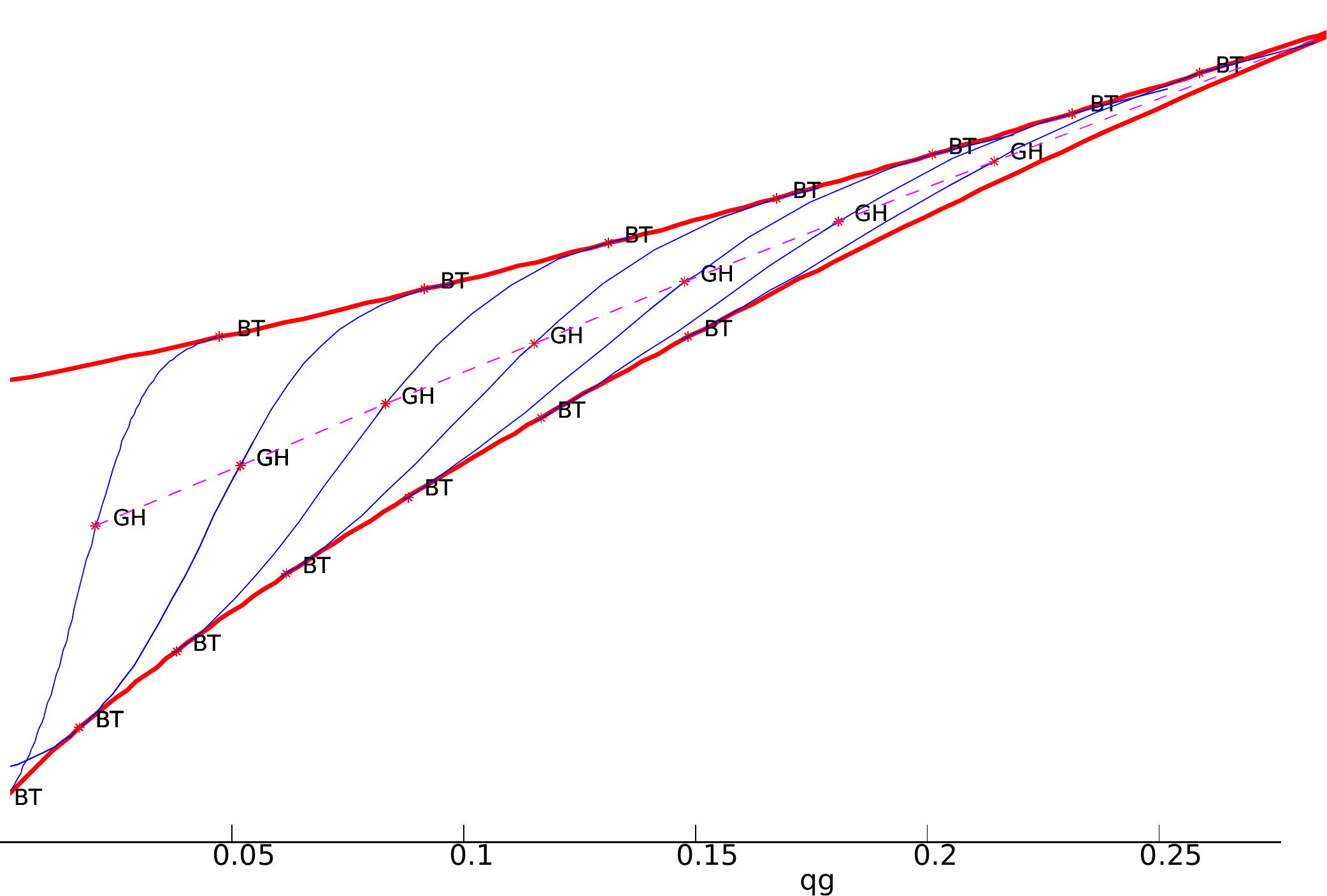}}\qquad
  \subfloat[]{\label{diagram1-right}\includegraphics[width=2.5in]{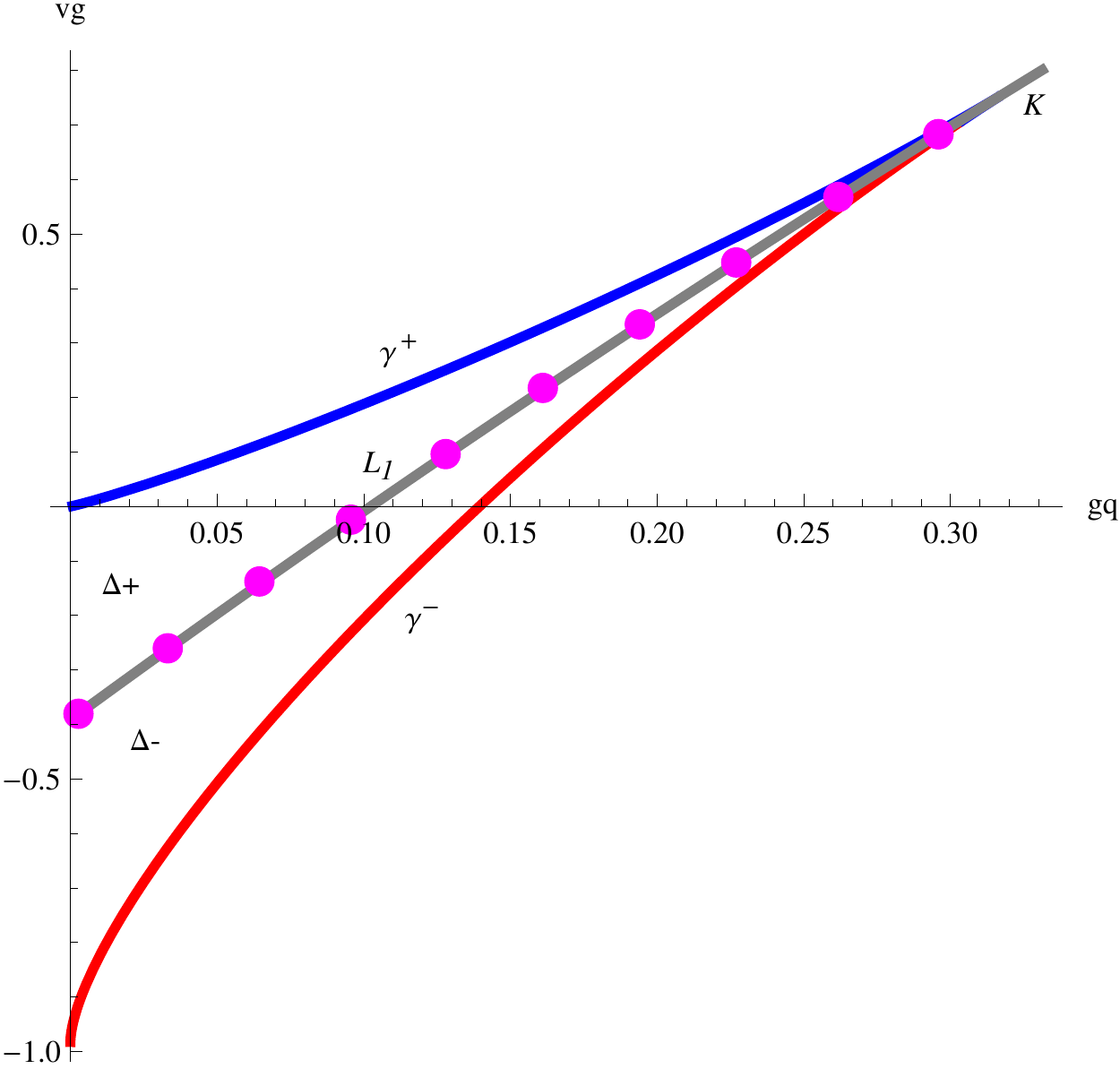}}
  \caption{Left: numerical continuation of Hopf curves from BT points. Right: Bautin points and the curve $\ell_1(q_g,v_g)=0.$ }\label{diagram1}
\end{figure}

 \subsection{Limit cycles}\label{homoclinics}
 Recall that the cuspidal region is partitioned in two components $\Delta^{\pm}$, the upper component $\Delta^{+}$ is defined by the boundaries $\gamma^{+}$  of  BT points and the curve $L_1$ of GH points where the first Lyapunov coefficient vanishes leading to Bautin bifurcations

The following analysis is performed on a particular  BT point in the lower part of the cuspidal curve. A similar analysis can be done with the other BT points.
Starting with this particular BT point,  we get a curve of Hopf points passing through a GH point. By further continuation,  we end up  with a BT  on the upper part of the cuspidal curve as it is shown in left graph of Figure~\ref{diagram1}. Next we take a Hopf point on one side of the GH point and perform the continuation of limit cycles holding the period fixed.

Examples of families of cycles of fixed period  in  parameter space $q_g$--$v_g$, for a fixed value of $\theta_0$, are shown in Sections~\ref{section-familyA} and~\ref{section-familyB}

 As the initial Hopf point is taken closer to the initial BT point, the period increases, in this way we obtain  a nested family  of curves of cycles of increasing period. These families tend  towards a limiting curve which is precisely the homoclinic curve of bifurcations emerging from the initial BT point.

\begin{figure}
\centering
\subfloat[]{\label{(a)} \includegraphics[width=2.5in]{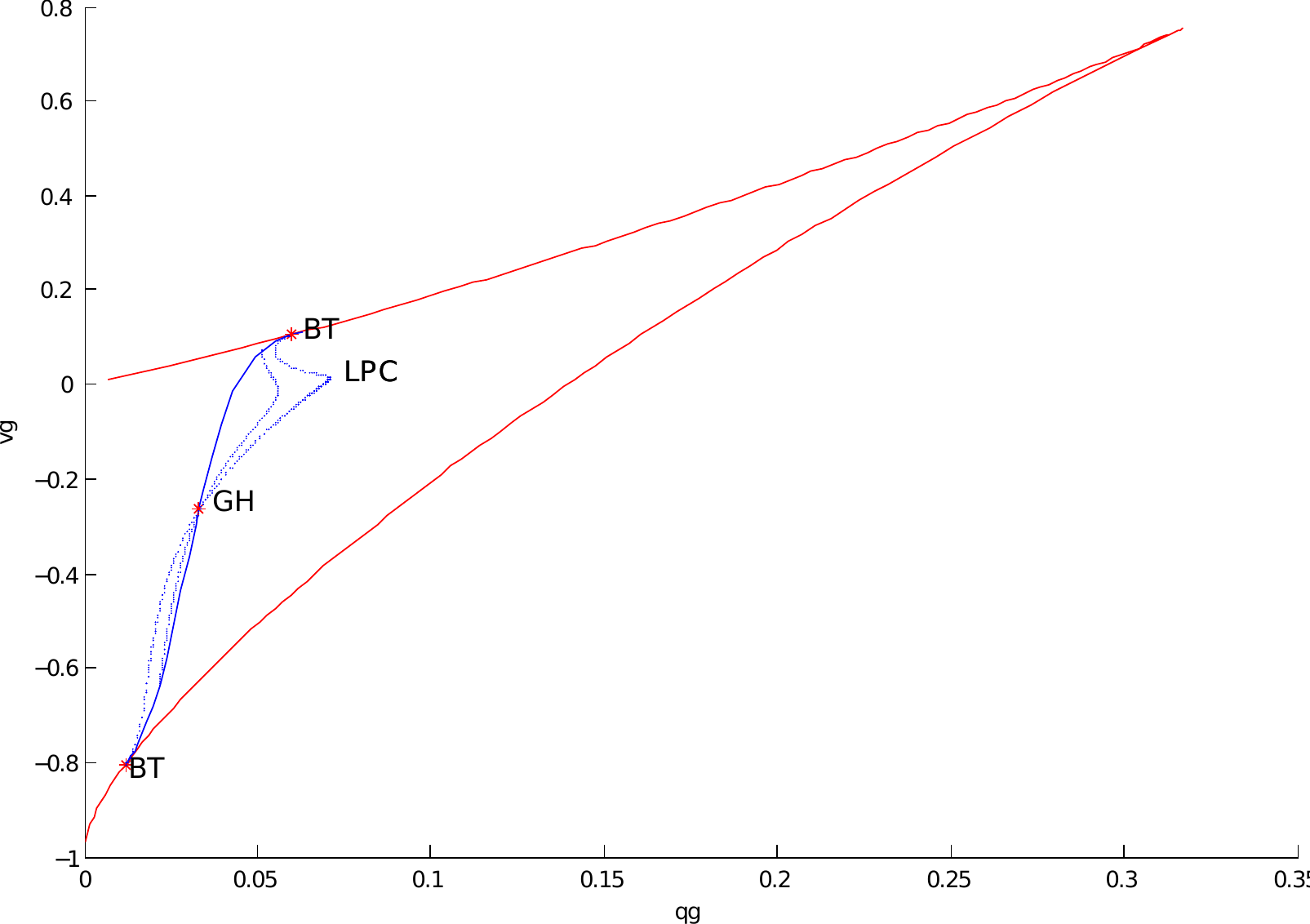}}
  \subfloat[]{\label{(a)}  \includegraphics[width=2.4in]{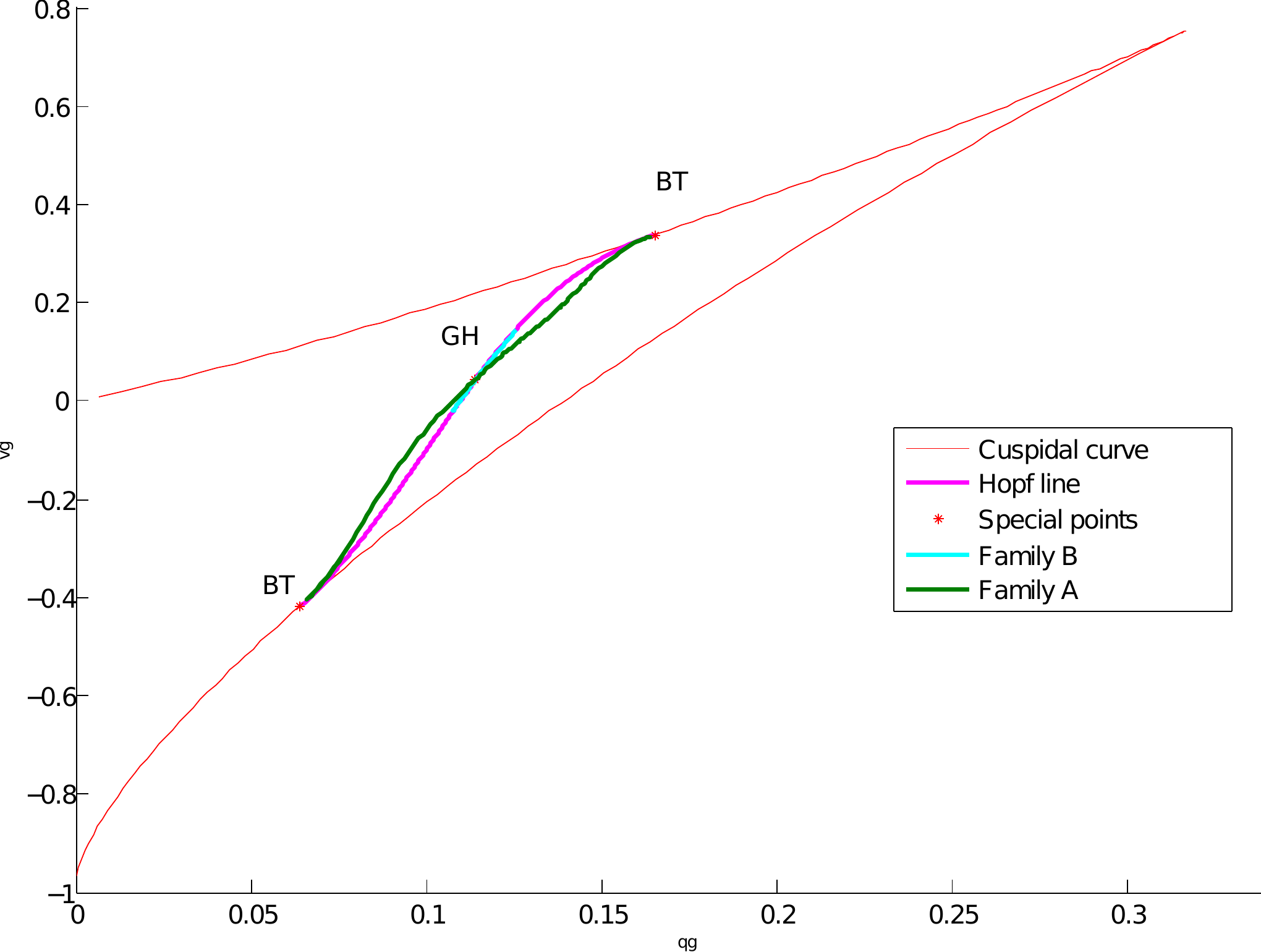}}
  \caption{(a) Two families of limit cycles of increasing period emerging from the line of Hopf points. These families accumulate towards the line of homoclinics.  LPC  is a turning point with respect to the parameter $q_g$.  (b) The two particular families: Families A (in green) of long period and Family  B (light blue) of short period. These families are presented in Sections~\ref{section-familyA} and  \ref{section-familyB}.
  }
\end{figure}

According to Corollary~1,  limit cycles located in the upper part of the cuspidal region are stable, while those in
the lower part are unstable.
 A natural question is if stable limit cycles correspond to stable traveling wave solutions of the PDE.

In order to explore this issue,  we take two limit cycles generated as explained  above, one in the stable region, the other  in the unstable region, as initial conditions for the PDE problem with periodic boundary conditions satisfying the condition (\ref{resonance}) with $m=1$, namely one--bump traveling waves.
We first illustrate the case of an unstable limit cycle  which gives place to an unstable traveling wave in  Figure  \ref{stable_unstable_solitions}. Here, the solution evolves towards a traveling wave, after a transient period.

\begin{figure}
\centering
 \includegraphics[width=3in]{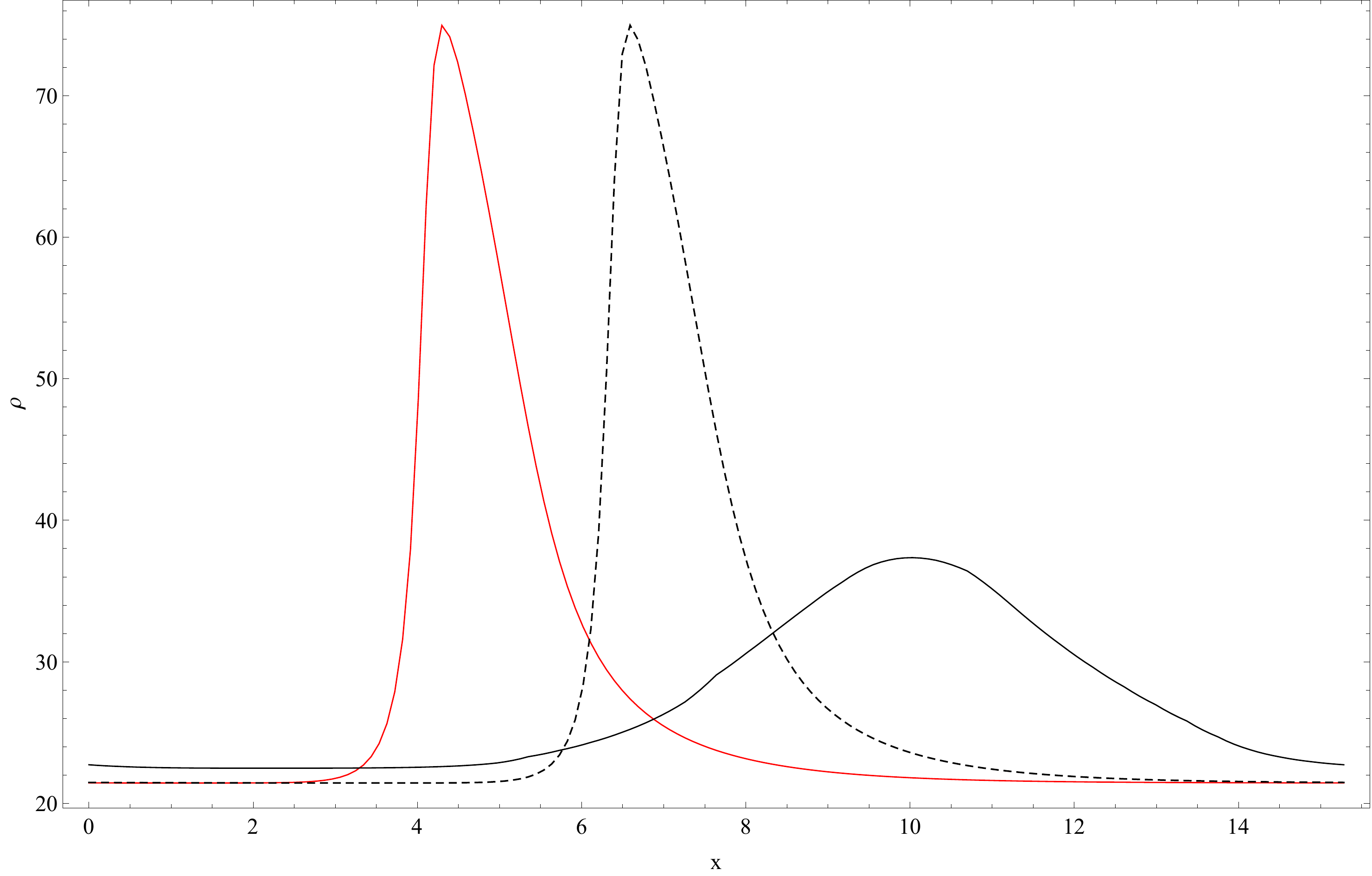}
  \caption{Unstable  traveling wave solution from an unstable limit cycle at $t=0$ min. (black continuous graph) and at $t=50$ (dashed graph), $80$ (red continuous graph) min., when the final shape is fully developed (right).}\label{stable_unstable_solitions}
\end{figure}

Further examples of stable limit cycles are presented in the form of families in the following sections.

\subsection{Family A of long period orbits}\label{section-familyA}
For this family we take initially the Hopf point
$$
q_g=0.164212226,\quad	v_g=0.335569670,\quad	v_c=0.064430330\quad	\theta_0=	0.16
$$
and continue into a family of stable limit cycles with period $T=1469.90$.
The value of the period correspond to a circuit of  length $L=10.49928571$~ km. The shape of some typical members of this family are shown in  Figure \ref{evolve}  left column.
We took the velocity and density profiles as initial conditions for the system of PDEs
(\ref{continuity}--\ref{balance}) and solved it numerically. In Figures \ref{(ev-fam3a)}, \ref{(ev-fam3b)},
\ref{(ev-fam3c)} we show the temporal evolution for the first 50 minutes, of some member of the family,  when a steady state solution of the PDE has fully developed.

\begin{figure}
\centering
 \subfloat[]{\label{(fam3)}\includegraphics[width=2in]{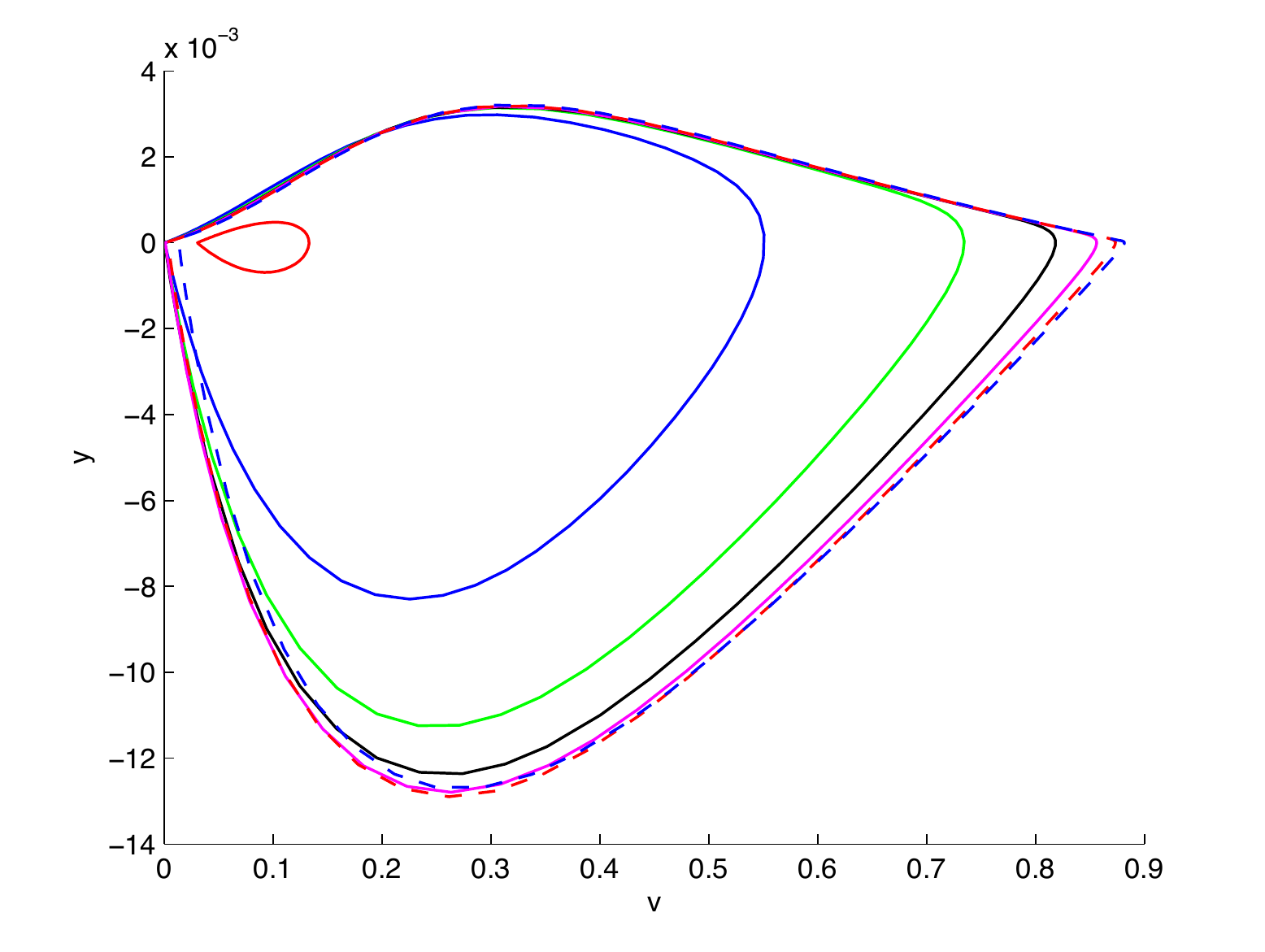}}\quad
   \subfloat[]{\label{(fam5)}\includegraphics[width=2in]{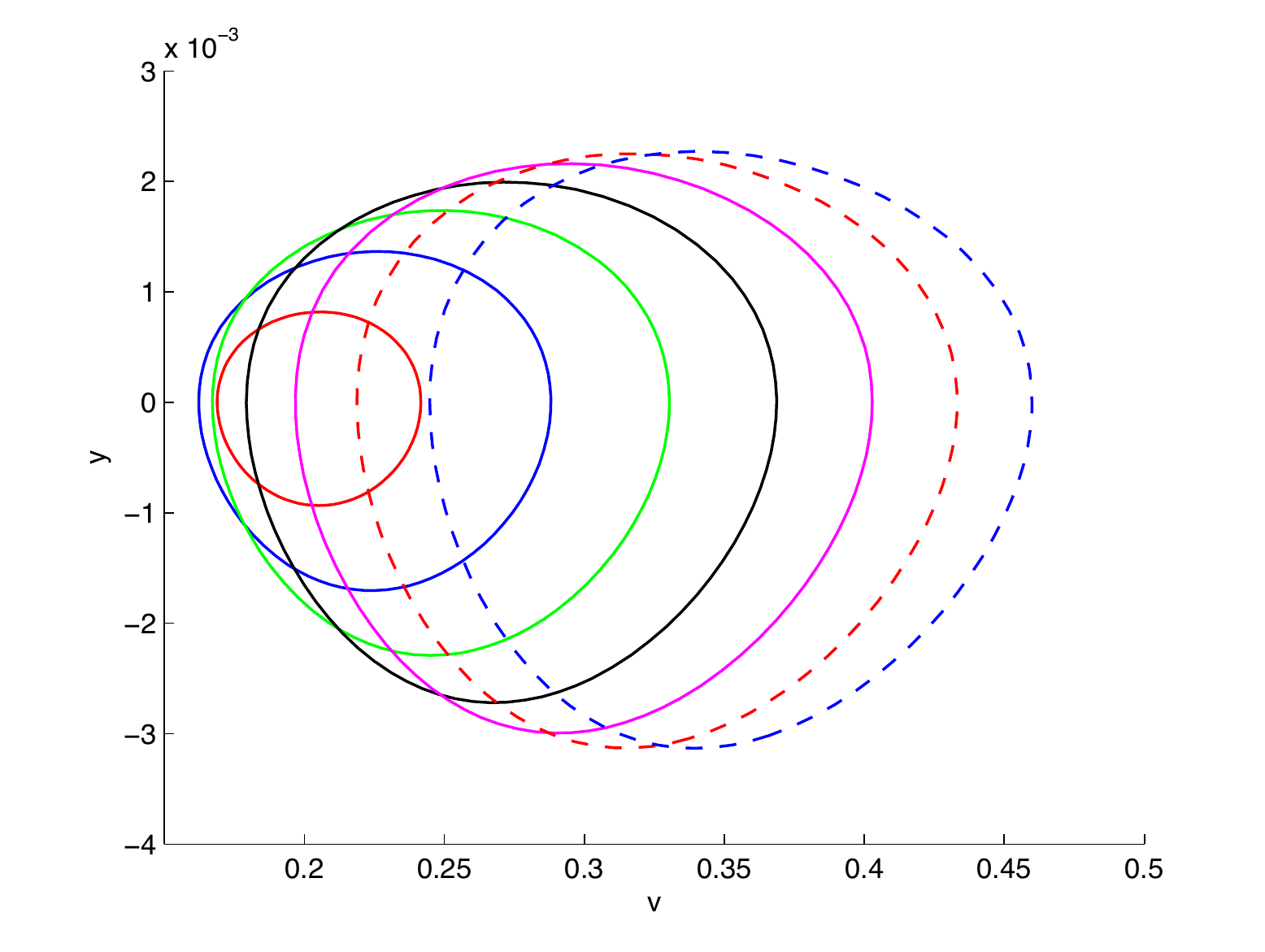}}\\

 \subfloat[]{\label{(ev-fam3a)}  \includegraphics[width=2in]{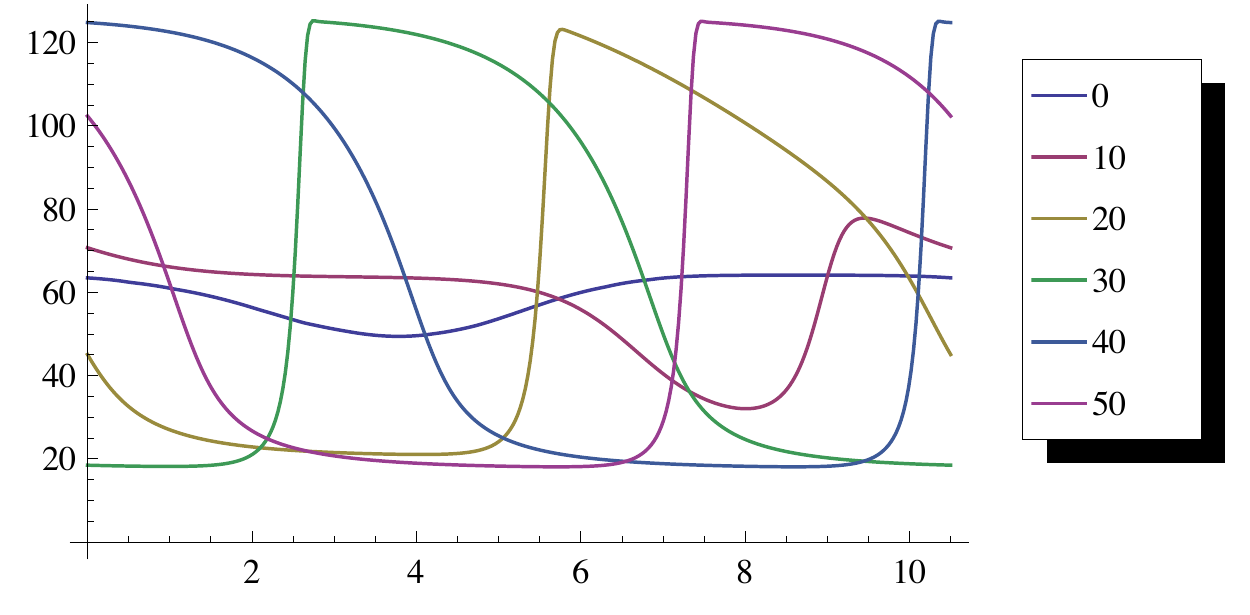}} \quad
         \subfloat[]{\label{(ev-fam5a)}  \includegraphics[width=2in]{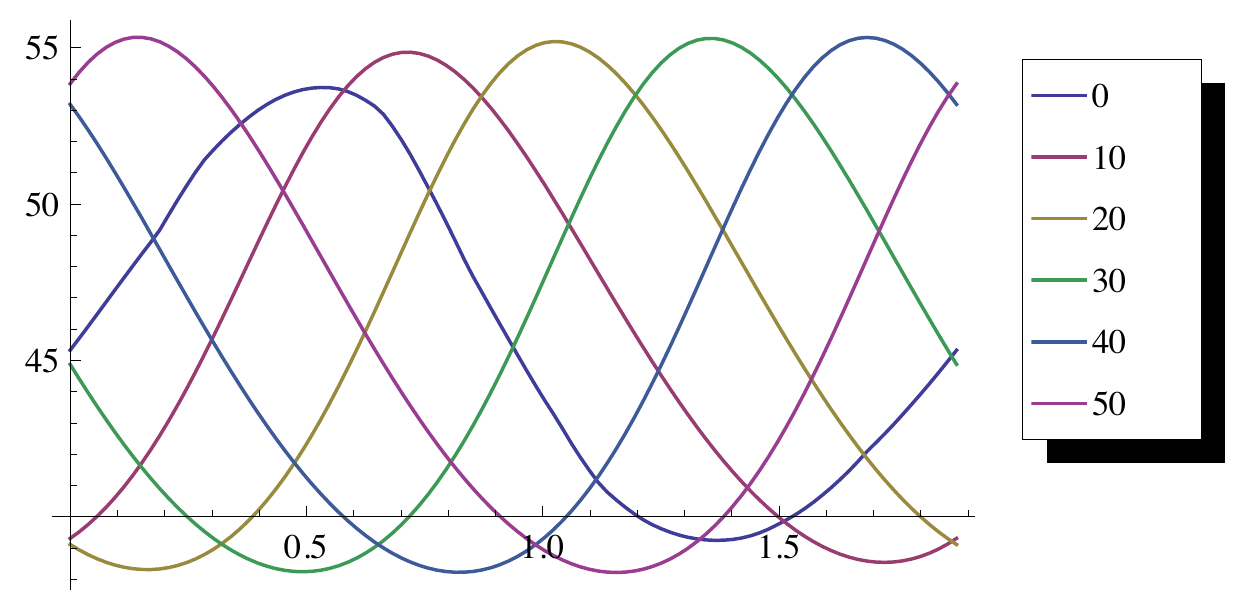}} \\
 \subfloat[]{\label{(ev-fam3b)}  \includegraphics[width=2in]{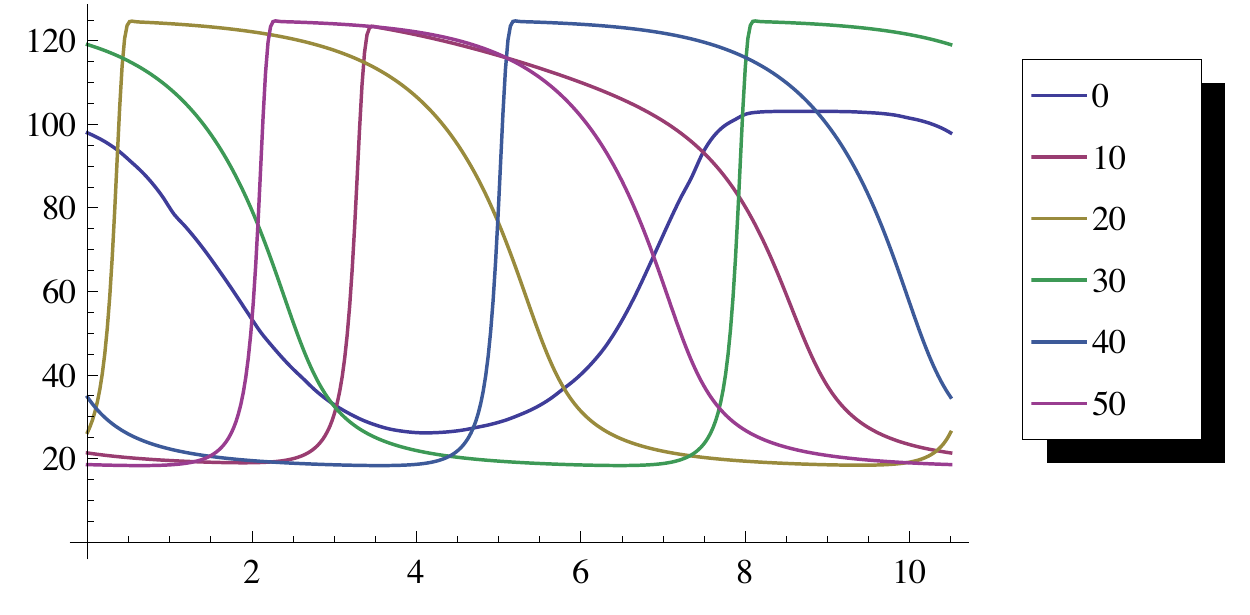}}\quad
            \subfloat[]{\label{(ev-fam5b)}\includegraphics[width=2in]{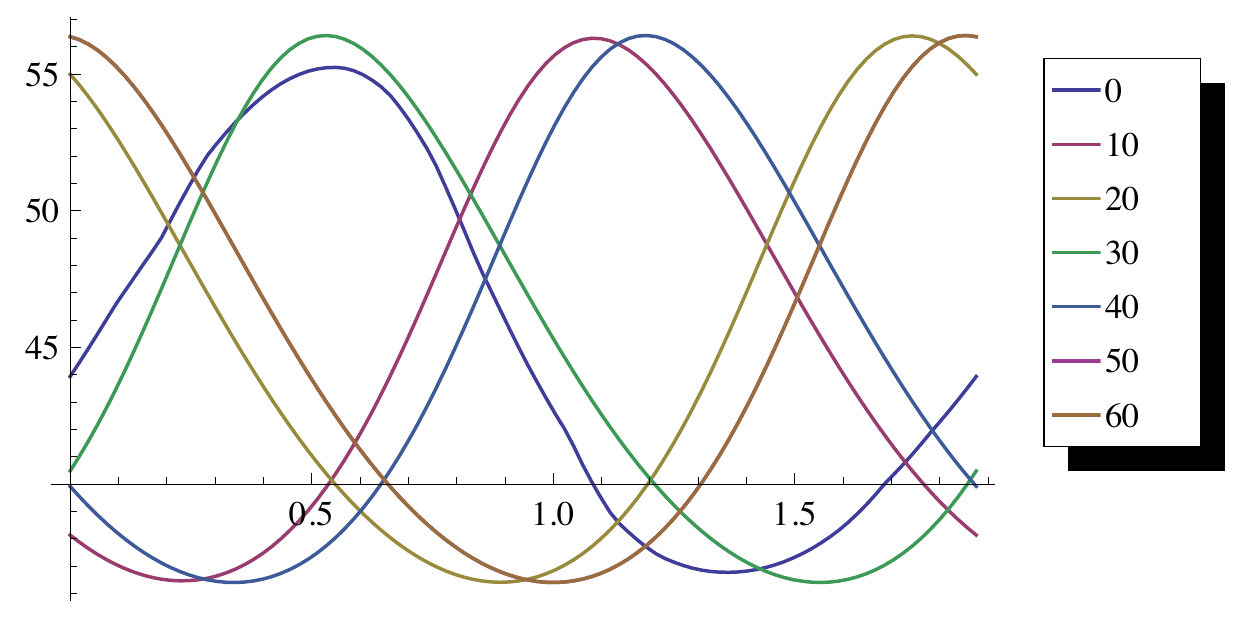}}\\
 \subfloat[]{\label{(ev-fam3c)}\includegraphics[width=2in]{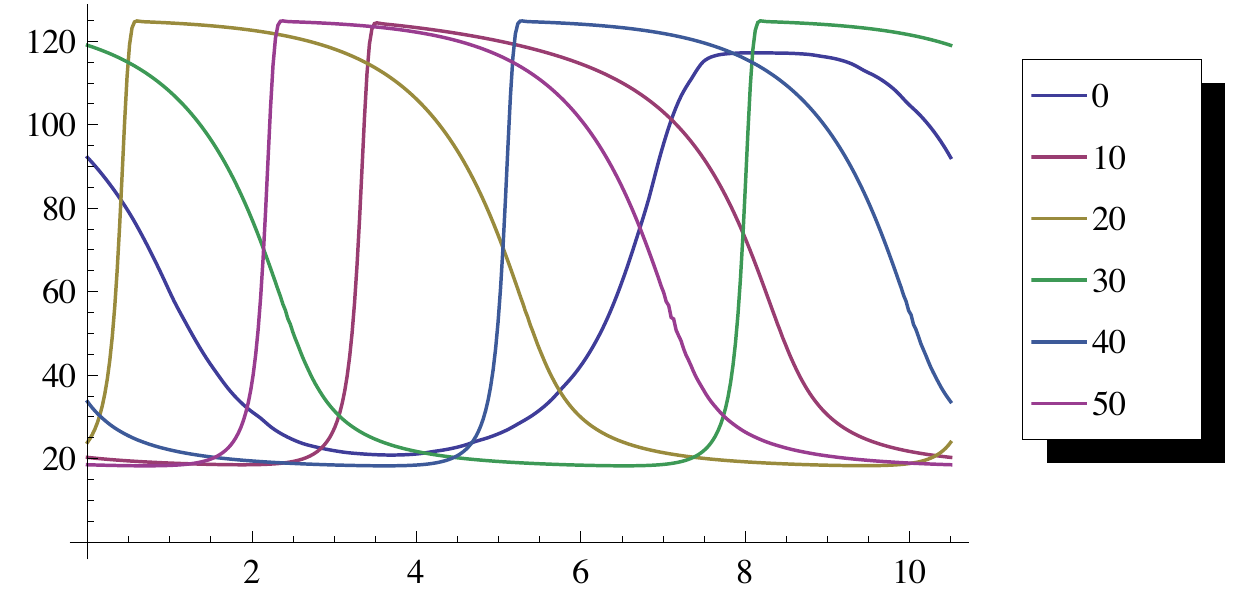}}\quad
         \subfloat[]{\label{(ev-fam5c)}\includegraphics[width=2in]{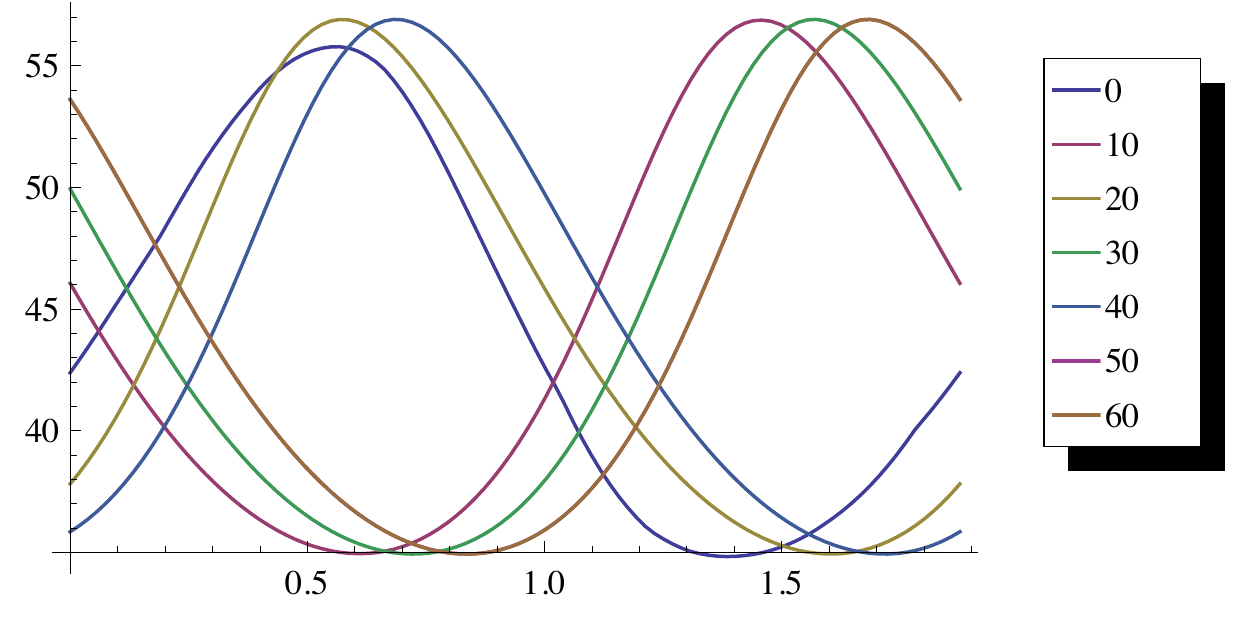}}
 \caption{
 Family A of long periodic cycles (left column) and family B of short periodic cycles  (right column)
 in phase space (first row). Temporal evolution in the PDE for some members of the family~A, for 10, 20, 30, 40, 50 min, and for some members of the family~B, for 10, 20, 30, 40, 50, 60 min. are shown in the following rows in the Figure.\label{evolve}}
\end{figure}

\subsection{Family B of short period orbits}\label{section-familyB}
Family B of short period orbits are shown in Figure~\ref{evolve}. It was computed by continuing to a stable limit cycle a Hopf point near the GH point in the stable part of the diagram. The values of the parameters of the
generating  Hopf point  of the family
are:
$$
q_g=0.133886021,\quad	v_g=0.204071932,\quad	v_c=0.195928068,\quad 	\theta_0=0.16.
$$
The period corresponds to a length of $2L=1.875802158$~km;
its characteristics are shown in Figure~\ref{evolve}.

\subsection{Stability of traveling waves}
The first Lyapunov coefficient determines the stability of limit cycles emerging from a Hopf bifurcation.  According to Proposition~\ref{ell1_positive},  the stability region of limit cycles in parameter space $q_g$--$v_g$ is the upper part $\Delta^+$ in Figure~\ref{diagram1}.

The relationship between Lyapunov stability of a limit cycle and the corresponding traveling wave solution is a delicate issue. Since the dynamical system (\ref{KKode}) is planar, from the Jordan closed curve theorem, a limit cycle defines a bounded region and an unbounded region in phase space. So
for example,  in the case of an unbounded road with bounded boundary conditions, a limit cycle may  be stable from the bounded region  and unstable from the outside part (as is the case of a saddle--node limit cycle), so one cannot assure the existence of a \emph{bounded solution} that is not completely contained in a neighborhood of the limit cycle. As another example, with the same kind of boundary conditions,  if a limit cycle is  unstable (from the bounded and unbounded regions), the corresponding traveling wave is unstable: this follows from the Poincar\'e-Bendixon theorem that guarantees the existence of a bounded solution inside the limit cycle, and from the very definition of Lyapunov instability of the limit cycle.

For the case of periodic boundary conditions. Neither instability of a limit cycle implies instability of the traveling wave, since for example an unstable limit cycle may contain in the bounded region an heteroclinic orbit connecting to a critical point, and by definition this heteroclinic does not satisfy periodic boundary conditions.

The two examples of families presented in the previous sections point out to the conjecture that stable limit cycles correspond to stable traveling waves. Limit cycles of  family A have a long period, and so are close to a homoclinic orbit, therefore they  spend a long time close to a critical point. This gives the family its sharp characteristic shown in Figure~\ref{(fam3)}. In particular our approximation of the limit cycle in 
 MatCont reveals not to be  precise enough to simulate the exact shape of the traveling wave, and therefore a short transient occurs  before the complete profile develops. This is becomes evident  for several member of the family  in Figures \ref{(ev-fam3a)}, \ref{(ev-fam3b)} and \ref{(ev-fam3c)} .  For family B, having short period, the numerical  approximation to the limit cycle with MatCont is good enough, as the initial profile at time $t=0$ is very similar to the fully developed profile. This behavior is shown in Figures~\ref{(ev-fam5a)}, \ref{(ev-fam5b)} and \ref{(ev-fam5c)}.

\subsection{Multiple bump traveling waves}
Multiple bump traveling waves are obtained by considering values of $m>1$. Thus for a limit cycle of minimal period $T$, there is one bump traveling wave in a road of length $L=T/\rho_{max}$ and a two
 bump traveling wave in a road of length $2L.$
In Figure~\ref{2bump}  we show a two--bump traveling wave obtained by the condition (\ref{resonance}) with $m=2$.
\begin{figure}
 \centering
\includegraphics[width=3in]{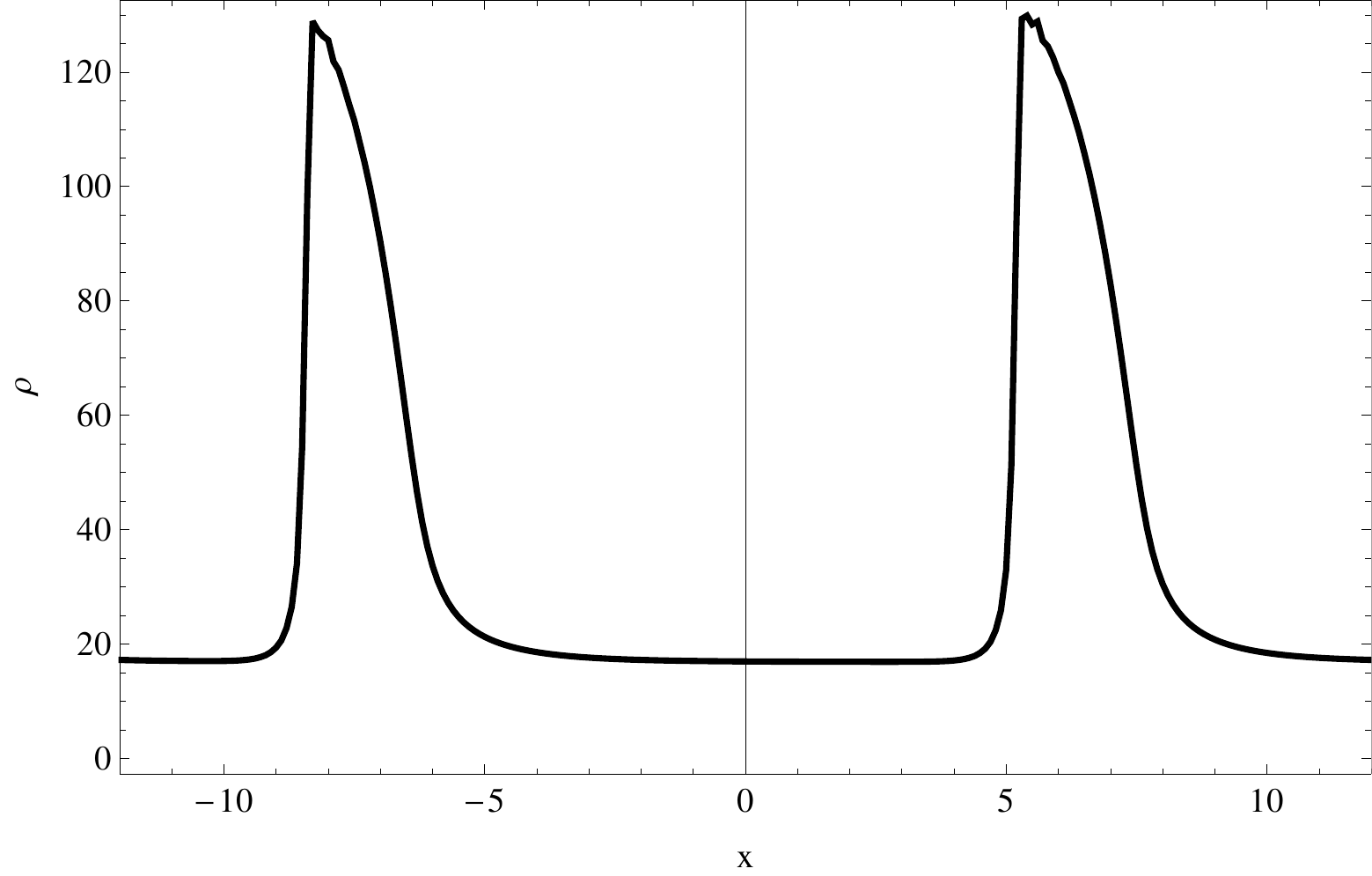}
\caption{A two-bump traveling wave. }
 \label{2bump}
\end{figure}

\section{Conclusions}\label{REMsection}
In this paper, we study traveling waves for the system of PDE (\ref{continuity}, \ref{balance})  for the Kerner--Konh\"auser fundamental diagram by the usual reduction to a system of ODE. We study the surface of critical points, and we analyzed thoroughly the cuspidal region in the parameter space $q_g$--$v_g$. We find, analytically and numerically, a complex map of  Hopf, Takens--Bogdanov, Bautin, homoclinics and heteroclinic bifurcations curves. This scenario is organized around a degenerate Takens Bogdanov point of bifurcation, according to the bifurcation diagram (\ref{lips})  due to Dumortier et al \cite{Du}.

Even though, there is a considerable simplification in the solution space, the dynamical system reveals the complexity of the space of solutions, which make us expect more complexity in the case of the PDE. Dynamical structures of the EDO system can tell us relevant things about the existence of periodic solutions in  bounded domains or bounded solutions in unbounded domains. In particular, limit cycles can be related to periodic solutions.  Homoclinic and heteroclinic trajectories describe traveling waves that tend to an homogenous solution when $\xi\to\pm\infty.$ Our numerical  results obtained in this work show that stable limit cycles yield stable traveling waves and viceversa. The non--linear stability is a more complicated issue which needs further study.

\appendix

\section{Proof of Proposition  \ref{DTB} }
We will write the dynamical system (\ref{KKode}) in normal form in order to analyze its coefficients and prove that it is a degenerate Takens Bogdanov point. Let $w_1=v-v_c$ and $w_2=y,$ then system (\ref{KKode}) is written as
\begin{eqnarray}\label{w12}
w_1'&=&f_1(w_1,w_2)=w_2, \\
w_2'&=& f_2(w_1,w_2) = \lambda q_g (1-\frac{\theta h^2}{(1+hw_1)^2}) w_2 -\mu q_g \frac{(v_e-w_1-v_c)}{(w_1+v_c +v_g)},
\end{eqnarray}
with  $h=\frac{1}{v_c+v_g}$.

By Hopf theorem, choosing $\theta$ as the reference parameter, if $\theta=\theta_0=(v_c+v_g)^2$ then $b(\theta_0)=0$ and the critical point $(v_c,0)$ has  imaginary eigenvalues $l_{1,2}=\pm i\omega_0$. Moreover, $$b'(\theta_0)=-\frac{\lambda q_g}{(v_c+v_g)^2}< 0$$
thus a limit cycle bifurcates from the critical point. Its stability relies on the sign of the first Lyapunov coefficient which we will explicitly calculate.
\bigskip

 Expanding in Taylor Series around $(0,0),$ we will write  system (\ref{KKode}) in the form
 $$
 \vec w'=A \vec w + \frac{1}{2}B(\vec w,\vec w) + \frac{1}{6}C(\vec w,\vec w,\vec w)+ \dots.
 $$
where the bilinear and trilinear forms are defined with $\xi=(\xi_1,\xi_2)$, $\eta=(\eta_1,\eta_2)$,
$\zeta=(\zeta_1,\zeta_2)$ as
\begin{equation}\label{B}
B(\xi,\eta)=\left(
\begin{array}{c}
\frac{\partial^2 f_1}{\partial w_1^2}\xi_1\eta_1 +\frac{\partial^2 f_1}{\partial w_1\partial w_2} (\xi_1\eta_2+\eta_1\xi_2)+ \frac{\partial^2 f_1}{\partial w_2^2}\xi_2\eta_2 \\[5pt]
\frac{\partial^2 f_2}{\partial w_1^2}\xi_1\eta_1 +\frac{\partial^2 f_2}{\partial w_1\partial w_2} (\xi_1\eta_2+\eta_1\xi_2) + \frac{\partial^2 f_2}{\partial w_2^2}\xi_2\eta_2
\end{array}
\right)
\end{equation}
and
$$
C(\xi,\eta,\zeta) =
$$
\begin{equation}\scriptsize
\left(
\begin{array}{c}
\frac{\partial^3 f_1}{\partial w_1^2}\xi_1\eta_1\zeta_1+
\frac{\partial^2 f_1}{\partial w_1^2 \partial w_2}( \xi_1\eta_1\zeta_2+\xi_2\eta_1\zeta_1+\xi_1\eta_2\zeta_1)+
\frac{\partial^2 f_1}{\partial w_1 \partial w_2^2} (\xi_1\eta_2\zeta_2+\xi_2\eta_2\zeta_1+\xi_2\eta_1\zeta_2)+
\frac{\partial^2 f_1}{\partial w_2^3} \xi_2\eta_2\zeta_2\\[5pt]
\frac{\partial^3 f_2}{\partial w_1^2}\xi_1\eta_1\zeta_1+
\frac{\partial^2 f_2}{\partial w_1^2 \partial w_2}( \xi_1\eta_1\zeta_2+\xi_2\eta_1\zeta_1+\xi_1\eta_2\zeta_1)+
\frac{\partial^2 f_2}{\partial w_1 \partial w_2^2} (\xi_1\eta_2\zeta_2+\xi_2\eta_2\zeta_1+\xi_2\eta_1\zeta_2)+
\frac{\partial^2 f_2}{\partial w_2^3} \xi_2\eta_2\zeta_2
\end{array}
\right)
\end{equation}

According to (\ref{w12}),
 $$
 \frac{\partial f_1}{\partial w_1}=0\quad\mbox{and}\quad
 \frac{\partial f_1}{\partial w_2}=1,
 $$
all the other higher order derivatives of $f_1$ are zero, thus  the first components of $B$ and $C$ are zero.
For the second components, we calculate the following partial derivatives
\begin{eqnarray*}
\frac{\partial f_2}{\partial w_1}&=&  \frac{2
\lambda q_g h^3 \theta}{(1+h w_1)^3}w_2-\frac{\mu q_g(v_e'-1)}{(w_1+v_c +v_g)}+ \frac{\mu q_g(v_e-w_1-v_c)}{(w_1+v_c+v_g)^2},\\
\frac{\partial f_2}{\partial w_2}&=& \lambda q_g \left(1-\frac{\theta h^2 }{(1+hw_1)^2}\right). \\
\end{eqnarray*}
The derivatives of second order are:
\begin{eqnarray*}
\frac{\partial^2 f_2}{\partial w_1^2}&=&
\frac{-6\lambda q_g h^4 \theta}{(1+h w_1)^4}w_2+\\
 && \frac{\mu q_g}{(w_1+v_c+v_g)}\left[\frac{ 2(v_e'-1)}{(w_1+v_c +v_g)}-v_e''- \frac{2(v_e-w_1-v_c)}{(w_1+v_c +v_g)^2}\right], \\
\frac{\partial^2 f_2}{\partial w_2 \partial w_1}&=&\frac{2 \lambda q_g h^3 \theta}{(1+h w_1)^3},\\
\frac{\partial^2 f_2}{\partial w_2^2}&=&0,
\end{eqnarray*}
when these derivatives are evaluated at  $w_1=w_2=0$  yields
\begin{eqnarray*}
\frac{\partial^2 f_2}{\partial w_1^2}
&=&
\frac{(2\omega_0^2-\mu q_g v_e'')}{(v_c +v_g)}.\\
\frac{\partial^2 f_2}{\partial w_2 \partial w_1}&=& 2 \lambda q_g h ,\\
\frac{\partial^2 f_2}{\partial w_2^2}&=&0.
\end{eqnarray*}

For the third order partial derivatives we get
\begin{eqnarray*}
\frac{\partial^3 f_2}{\partial w_1^3}&=&
\frac{24 \lambda q_g h^5 \theta}{(1+h w_1)^5}w_2+\frac{\mu q_g}{(w_1+v_c+v_g)}\cdot\\
&& \left[\frac{ 3v_e''}{(w_1+v_c +v_g)}-v_e'''-\frac{ 6(v_e'-1)}{(w_1+v_c +v_g)^2}
+\frac{6(v_e-w_1-v_c)}{(w_1+v_c+v_g)^3}\right],
\end{eqnarray*}
and
\begin{eqnarray*}
\frac{\partial^3 f_2}{\partial w_2 \partial w_1^2}&=&-\frac{6\lambda q_g h^4\theta}{(1+hw_1)^4},\\
\frac{\partial^3 f_2}{\partial w_2^2 \partial w_1}&=&\frac{\partial^3 f_2}{\partial w_2^3}=0.
\end{eqnarray*}
when they are evaluated in $w_1=w_2=0,$ we obtain
\begin{eqnarray*}
\frac{\partial^3 f_2}{\partial w_1^3}&=&
\frac{1}{(v_c+v_g)}\left[\frac{ 3\mu q_g v_e''}{(v_c +v_g)}-\mu q_gv_e'''-\frac{ 6\omega_0^2}{(v_c +v_g)}
\right]\\
\frac{\partial^3 f_2}{\partial w_2 \partial w_1^2} &=& -6\lambda q_g h^4\theta.
\end{eqnarray*}

Then $B$ and $C$ are equal to
$$B(\xi,\eta)=
\left(
\begin{array}{l}
 0  \\
\frac{(2 \omega_0^2-\mu q_g v_e''(v_c))}{(v_c+v_g)}\xi_1 \eta_1+2\lambda q_g h (\xi_1 \eta_2+\eta_1\xi_2 )
\end{array}
\right)
$$
and\goodbreak
$$C(\vec \xi, \vec \eta,\vec \zeta)=$$
$$\scriptsize
\left(
\begin{array}{l}
 0  \\
\frac{1}{(v_c+v_g)}\left[\frac{3 \mu q_gv_e''(v_c)}{(v_c+v_g)}-\mu q_g v_e'''(v_c)-\frac{6\omega_0^2}{v_c+v_g}\right]\xi_1 \eta_1 \zeta_1 -6\lambda q_g h^4\theta( \xi_1\eta_1\zeta_2+\xi_2\eta_1\zeta_1+\xi_1\eta_2\zeta_1)
\end{array}
\right).
$$

To calculate the first Lyapunov coefficient we have first to calculate  vectors $\vec q$ and $\vec p$ such that
$A \vec q=\omega_0 i \vec q$ and $A^T \vec p=-\omega_0 i \vec p,$ respectively, and they satisfy $\langle \vec p, \vec q \rangle=1$. We take
 $\vec q^T=(1, \omega_0 i)$ and $ \vec p^T=\frac{1}{2}(1,\frac{i}{\omega_0})$.
Now we have to calculate
$g_{20}=\langle \vec p, B(\vec q, \vec q) \rangle,$ $g_{11}=\langle \vec p, B(\vec q, \vec{\overline{q}} ) \rangle$ and $g_{21}=\langle \vec p, C(\vec q, \vec q, \vec{\overline{q}} )\rangle$ in order to evaluate
\begin{equation}\label{l1}
\ell_1=\frac{1}{2 \omega_0^2} Re(i g_{20} g_{11}+\omega_0 g_{21}),
\end{equation}
which is the first Lyapunov coefficient. Now,
\begin{eqnarray*}
g_{20}&=&
2 \lambda q_g h-\frac{(2\omega_0^2-\mu q_g v_e''(v_c)) i}{2\omega_0(v_c+v_g)},
\end{eqnarray*}

\begin{eqnarray*}
g_{11}&=&
-\frac{(2\omega_0^2-\mu q_g v_e''(v_c)) i}{2\omega_0(v_c+v_g)},
\end{eqnarray*}

\begin{eqnarray*}
g_{21}&=&
\frac{i}{2\omega_0(v_c+v_g)}\left[\frac{3\mu q_g v_e''(v_c)}{(v_c+v_g)}-\mu q_g v_e'''(v_c)
-\frac{6\omega_0^2}{(v_c+v_g)}\right]-3\lambda q_g h^4\theta
\end{eqnarray*}
Thus
\begin{eqnarray*}
ig_{20}g_{11}&=&
\frac{ \lambda q_g h(2\omega_0^2-\mu q_g v_e''(v_c))}{\omega_0(v_c+v_g)} -\frac{(2\omega_0^2-\mu q_g v_e''(v_c)) i}{2\omega_0(v_c+v_g)}.
\end{eqnarray*}
Substituting these values in (\ref{l1}) we obtain
\begin{eqnarray*}
l_1(\theta_0)
&=&-\frac{\lambda\mu q_g^2 h}{2\omega_0^3(v_c+v_g)}\left( \frac{v_e'(v_c)-1}{v_c+v_g}+v_e''(v_c)\right).
\end{eqnarray*}

\section{Proof of theorem \ref{DTB}}

Expanding in Taylor Series $c(w_1)=\lambda q_g (1-\frac{\theta h^2}{(1+hw_1)^2})$ and $f(w_1)=\mu q_g \frac{(v_e(v)-w_1-v_c)}{(w_1+v_c +v_g)}=L(w_1)(v_e-w_1-v_c)$ around $(0,0)$  we obtain:
$$(1+hw_1)^2 =(1-2h w_1+3 h^2 w_1^2-4 h^3 w_1^3+5h^4 w_1^4+\dots  ) $$
and
$$
c(w_1)=\lambda q_g (1-\theta h^2(1-2h w_1+3 h^2 w_1^2-4 h^3 w_1^3+5h^4 w_1^4+\dots  )).
$$
If we chose $\theta_0=(v_c+v_g)^2$ then $\theta_0 h^2=1$ and
$$
c(w_1)=\lambda q_g \theta_0 h^3( 2 w_1-3 h^2 w_1^2+4 h^3 w_1^3-5h^4 w_1^4+\dots  ).
$$
Then $$c(w_1)w_2= \lambda q_g \theta_0 h^3 ( 2 w_1 w_2-3 h^2 w_1^2 w_2+....)=b_2 w_1 w_2+ b_3w_1^2 w_2+\dots.$$
where
$$
b_2=2\lambda q_g h,\quad b_3= -3\lambda q_g h^3.
$$

On the other hand,
\begin{equation}
f(w_1)=f(0)+f'(0) w_1+ \frac{1}{2} f''(0) w_1^2 + \frac{1}{6}f'''(0) w_1^3+\dots
\end{equation}
with
\begin{eqnarray*}
f'(w_1)&=&L'(w_1) (v_e-w_1-v_c)+L(w_1)(v_e'-1), \\
f''(w_1)&=&L''(w_1) (v_e-w_1-v_c)+2L'(w_1)(v_e'-1)+L(w_1) v_e'', \\
f'''(w_1)&=&L'''(w_1)(v_e-w_1-v_c)+3L''(w_1)(v_e'-1)+3L'(w_1)v_e''+L(w_1) v_e'''.  \\
\end{eqnarray*}
Evaluating these derivatives in $w_1=0$ and using the hypothesis we obtain
$$
f(w_1)= \frac{1}{6}f'''(0) w_1^3 +\dots=a_3 w_1^3+a_4 w_1^4+\dots
$$
Given that $a_2=0$ and $a_3 b_2 \neq 0$ we can write system (\ref{w12})
in the normal form (\ref{orbitally-equivalent}) as
\begin{eqnarray*}
\dot{w}_0&=&w_1, \\
\dot{w}_1&=& a_3  w_0^3 + b_2 w_0 w_1 + b'_3 w_0^2 w_1+O(\|(w_0,w_1)\|)^5 . \\
\end{eqnarray*}
where $a_3=\frac{-\mu q_g v_e'''(v_c)}{6(v_c+v_g)},$ and $b'_3=b_3-\frac{3b_2 a_4}{5a_3}.$
By hypothesis  $v_e'''(v_c)<0,$ therefore  $a_3>0,$ and  we are in the saddle case.

\end{document}